\documentclass[12pt]{iopart}

\expandafter\let\csname equation*\endcsname\relax
\expandafter\let\csname endequation*\endcsname\relax

\usepackage[T1]{fontenc}
\usepackage[latin9]{inputenc}
\usepackage{geometry}
\usepackage{xcolor}
\usepackage{array}
\usepackage{rotfloat}
\usepackage{calc}
\usepackage{units}
\usepackage{multirow}
\usepackage{amsmath}
\usepackage{amsthm}
\usepackage{amssymb}
\usepackage{graphicx}
\usepackage{setspace}
\usepackage[numbers]{natbib}
\PassOptionsToPackage{normalem}{ulem}
\usepackage{ulem}
\makeatletter

\providecommand{\tabularnewline}{\\}
\newcommand{\lyxdot}{.}

\providecolor{lyxadded}{rgb}{0,0,1}
\providecolor{lyxdeleted}{rgb}{1,0,0}

\DeclareRobustCommand{\lyxdeleted}[3]{{\color{lyxdeleted}\lyxsout{#3}}}
\DeclareRobustCommand{\lyxsout}[1]{\ifx\\#1\else\sout{#1}\fi}
\usepackage{tikz}
\usetikzlibrary{calc}

\theoremstyle{plain}
\newtheorem{thm}{\protect\theoremname}
\theoremstyle{plain}
\newtheorem{cor}[thm]{\protect\corollaryname}
\theoremstyle{plain}
\newtheorem{lem}[thm]{\protect\lemmaname}
\theoremstyle{plain}
\newtheorem*{thm*}{\protect\theoremname}
\theoremstyle{definition}
\newtheorem{defn}[thm]{\protect\definitionname}

\usepackage{diagbox}

\makeatother

\usepackage[english]{babel}
\providecommand{\corollaryname}{Corollary}
\providecommand{\lemmaname}{Lemma}
\providecommand{\theoremname}{Theorem}
\providecommand{\definitionname}{Definition}

\begin{document}

\maketitle

\title[Numerical Error in Inverse Problems for PDEs]{The Influence of Numerical Error on an Inverse Problem Methodology in PDE Models}

\author{John T. Nardini$^{1,2}$, D. M. Bortz$^1$}

\address{$^1$ Department of Applied Mathematics,
526 UCB,
Boulder, CO 80309-0526}

\address{$^2$ Interdisciplinary Quantitative Biology Graduate Program,
University of Colorado,
Boulder, CO United States 80309-0596}

\ead{john.nardini@colorado.edu, dmbortz@colorado.edu}

\begin{abstract}
The inverse problem methodology is a commonly-used framework in the
sciences for parameter estimation and inference. It is typically performed
by fitting a mathematical model to noisy experimental data. There
are two significant sources of error in the process: 1.$\ $Noise
from the measurement and collection of experimental data and 2.$\ $numerical
error in approximating the true solution to the mathematical model.
Little attention has been paid to how this second source of error
alters the results of an inverse problem. As a first step towards
a better understanding of this problem, we present a modeling and
simulation study using a simple advection-driven PDE model. We present
both analytical and computational results concerning how the different
sources of error impact the least squares cost function as well as
parameter estimation and uncertainty quantification. We investigate
residual patterns to derive an autocorrelative statistical model that
can improve parameter estimation and confidence interval computation
for first order methods. Building on the results of our investigation,
we provide guidelines for practitioners to determine when numerical
or experimental error is the main source of error in their inference,
along with suggestions of how to efficiently improve their results.
\end{abstract}


\section{Introduction}

Differential Equations are frequently used to study scientific systems.
When there are multiple independent variables influencing this system
(such as time and space), then a partial differential equation (PDE)
is the appropriate modeling framework. Due to their complicated nature,
deriving an analytical solution to a PDE model is frequently difficult
or impossible, so scientists must use numerical methods to approximate
the true solution. How the error from this approximation influences
some aspects of an inverse problem, such as parameter estimation and
uncertainty quantification, is an important and poorly-understood
problem. We will focus on a deterministic inverse problem in this study, but approximation
errors in determining the likelihood are of significant concern in
Bayesian methods as well ~\citep{holmes_practical_2015}. There have
been some notable previous efforts to elucidate these questions. For example,
Banks and Fitzpatrick ~\citep{banks_statistical_1990} proved the asymptotic
consistency of the parameter estimator for least squares estimation
in the presence of numerical approximation error, and Xue \emph{et
al.} ~\citep{xue_sieve_2010} derived the asymptotic distribution of
this estimator when a numerical approximation is used to for an ordinary
differential equation model.

In a previous study on parameter estimation, Ackleh and Thibodeaux
~\citep{ackleh_parameter_2008} consider an advection-driven model
of erythyropoiesis (an important step in red blood cell development)
with three independent variables of time, maturity, and space. The
authors show that using an upwind scheme for computation during an
inverse problem is asymptotically well-posed for parameter estimation
as the numerical step size used, $h$, approaches zero. In practice,
however, one cannot let $h$ approach zero but must choose a finite
value of $h$ to estimate the parameters with. Furthermore, advection
equations such as that used in ~\citep{ackleh_parameter_2008} are
known to cause a multitude of numerical issues, especially when the
true solution is discontinuous ~\citep{leonard_ultimate_1991,randall_j._leveque_conservation_1992,thackham_computational_2008}.
The upwind method is a popular choice to simulate these problems because
it can avoid spurious oscillations by satisfying the Courant-Friedrichs-Lewy
(CFL) condition, but this method also causes its own difficulties
by admitting numerical diffusion near points of discontinuity ~\citep{leveque_finite_2007,randall_j._leveque_conservation_1992}.

The work in these publications raise a multitude of questions to consider
when performing an inverse problem. How do the least squares cost
function and parameter estimator behave as numerical error decreases?
For the step size used, what is the dominant source of error in the
cost function computation? If numerical error dominates, can we use
residuals as a way to update how we compute the cost function and
improve the inverse problem's results? How do we know if the chosen
numerical method can accurately estimate parameters? How do the results
change based on properties of the model's solution?

In this study, we will use a simple advection equation to demonstrate
the impact of numerical error from several finite difference and finite volume methods
on an inverse problem methodology. To compare the influence of numerical
versus experimental error in this study, we will fit these computations
to data sets that have been artificially generated from the analytical
solution with varying levels of experimental noise. We begin in Section
~\ref{sec:Mathematical-Preliminaries} by introducing some preliminary
information, including the equation under consideration and its analytical
solution, how we generate the artificial data sets, and the numerical
methods used in this study. In Section ~\ref{sec:Numerical-Cost-Function},
we introduce the inverse problem methodology and discuss the asymptotic
results for the parameter estimator and numerical cost function used
in this framework. In Section ~\ref{sec:Inverse-Problem-Results},
we discuss our results in using these numerical methods to estimate
parameters from the data sets. We use residual analysis in Section
~\ref{sec:Residual-Analysis} to demonstrate how numerical error from
first order numerical methods leads to an autocorrelated error structure
when comparing the model to data. To address this issue, we derive
an autocorrelative statistical model to describe how this numerical
error propagates throughout the inverse problem. We further demonstrate
how this autocorrelated statistical model can be used to improve confidence
interval computation and parameter estimation. Based on these results,
we provide some suggestions and guidance for practitioners in Section
~\ref{sec:Suggestions-for-practitioners} to ensure that the results
of their inverse problem routines are as accurate as possible. We
make concluding remarks and discuss future work in Section ~\ref{sec:Discussion-and-Future}.

\section{Mathematical Preliminaries \label{sec:Mathematical-Preliminaries}}

In this section, we detail some necessary information regarding our
inverse problem methodology. We discuss the advection equation and
choice of parameterization in Section ~\ref{subsec:PDE-Model-Equation}.
In Section ~\ref{subsec:Explanation-of-Notation}, we will present
some notation used throughout this work. We present how we generate
artificial data for this study in Section ~\ref{subsec:Artificial-Data-Generation}.
In Section ~\ref{subsec:Numerical-Methods}, we discuss the numerical
schemes that we will use in this study. 

\subsection{PDE Model Equation\label{subsec:PDE-Model-Equation}}

We will consider an advection equation in one spatial dimension. We
define our spatial domain as $X=[0,1]$, the temporal domain as $T=[0,10],$
and the parameter value domain as $\Theta=\mathbb{R}^{k_{\theta}}$
for $k_{\theta}$ denoting the number of parameters to be estimated.
The advection equation is given by

\begin{align}
u_{t}+(g(x;\theta)u)_{x} & =0,\ u=u(t,x;\theta)\label{eq:adv_eqn}\\
u(t=0,x;\theta) & =\phi(x)\nonumber \\
x\in X,t\in T & ,\theta\in\Theta\nonumber 
\end{align}

where subscripts denote differentiation, $g(x;\theta)$ is a spatially-dependent
advection rate that is parameterized by the vector $\theta$, $\phi(x)$
is the initial condition, and $u(t,x;\theta)$ denotes the quantity
of interest at time $t$ and spatial location $x$ that is also parameterized
by $\theta$. We will suppress the dependence of $g(x)$ and $u(t,x)$
on $\theta$ throughout this study when this dependence can be implicitly
understood.

The method of characteristics can be used to show the analytical solution
to Equation (\ref{eq:adv_eqn}) of

\begin{equation}
u_{0}(t,x)=\begin{cases}
\dfrac{g(\sigma^{-1}(-t,x))}{g(x)}\phi(\sigma^{-1}(-t,x)) & \sigma^{-1}(t,0)\le x\le1\\
0 & \text{otherwise}
\end{cases}\label{eq:adv_analtical}
\end{equation}
where $\sigma^{-1}(t,x)$ is the characteristic curve that satisfies
the initial value problem
\begin{align*}
\dfrac{\partial}{\partial t}\sigma^{-1}(t,x)=g\left(\sigma^{-1}(t,x)\right),\ \sigma^{-1}(t=0,x)=x.
\end{align*}
See ~\citep{webb_population_2008} for more information about deriving
this analytical solution and ~\citep{nardini_investigation_2018} for
an illustrative example of this concept involving biochemical activation
during wound healing. We choose the rate of advection
\begin{align*}
g(x)=\alpha\sqrt[\beta]{x},\ \ \ \alpha,\beta>0
\end{align*}
for $\theta=(\alpha,\beta)^{T}\in\Theta=\mathbb{R}^{2}$. The choice
of $g(x)$ above yields the characteristic curves 
\begin{align*}
\sigma^{-1}(t,x)=\left[\alpha(1-\nicefrac{1}{\beta})t+x^{1-\nicefrac{1}{\beta}}\right]^{\nicefrac{\beta}{(\beta-1)}}.
\end{align*}

We will consider two initial conditions in this study to demonstrate
how spatial continuity influences numerical convergence and the inverse
problem results. To demonstrate the behavior for a discontinuous solution,
we will focus on simulations with a \emph{discontinuous} initial condition
given by the step function 
\begin{align*}
\phi_{d}(x)=\begin{cases}
5 & x\le0.2\\
0 & \text{otherwise}
\end{cases}
\end{align*}
To illustrate how the results change for a \emph{continuous} solution,
we will also consider the Gaussian-shaped initial condition given
by
\begin{align*}
\phi_{c}(x)=e^{-\left(\nicefrac{x-0.2}{\sqrt{.005}}\right)^{2}}.
\end{align*}
We will focus on the results for $\phi(x)=\phi_{d}(x)$ in the main
body of this document, with the corresponding results for $\phi(x)=\phi_{c}(x)$
in the supporting material. We will make note of how the results change
between these two initial conditions when appropriate.

\subsection{Explanation of Notation\label{subsec:Explanation-of-Notation}}

Note that throughout this work, $M$ and $N$ denote the number of
time and spatial points provided in an artificial data set, respectively.
We will denote the numerical step size as $h$, which will determine
the number of grid points used during numerical computations. It is
important to realize that $h$, $M$, and $N$ are all independent
of one another.

Our data sets will be provided on the uniform partitions of $T\times X$
given by $T^{M}\times X^{N}$, where
\begin{align*}
T^{M}=\{t_{i}\}_{i=1}^{M}=10\left\{ \nicefrac{(i-1)}{M}\right\} _{i=1}^{M},\ \ X^{N}=\{x_{j}\}_{j=1}^{N}=\left\{ \nicefrac{j-1}{N}\right\} _{j=1}^{N}.
\end{align*}
We will write a given data set as the $M\times N$ matrix, $Y$. The
$(i,j)th$ entry of $Y$ is given by
\begin{align*}
\left[Y\right]_{i,j}=y_{i,j},
\end{align*}
where $y_{i,j}$ denotes the observation of the data at time $t_{i}$
and location $x_{j}.$ We will denote the analytical solution to Equation
(\ref{eq:adv_eqn}) with parameter value $\theta$ on $T^{M}\times X^{N}$
as the $M\times N$ matrix, $U_{0}(\theta)$, with $(i,j)th$ entry
\begin{align*}
\left[U_{0}(\theta)\right]_{i,j}=u_{0}(t_{i},x_{j};\theta).
\end{align*}
We will denote a numerical computation that has been computed with
numerical step size $h$ and parameter value $\theta$ and then interpolated\footnote{Note that the interpolation step is performed with an $\mathcal{O}(h^{3})$
procedure while the finite difference schemes are $\mathcal{O}(h^{p})$
for $p\le2.$ This interpolation step should thus not alter other
convergence rates.} to $T^{M}\times X^{N}$ as the $M\times N$ matrix $U(h,\theta)$
with $(i,j)th$ entry
\begin{align*}
\left[U(h,\theta)\right]_{i,j}=u(t_{i},x_{j};h,\theta).
\end{align*}
Arrows on top of these sets will denote their vectorizations, \emph{e.g.},
the vector $\vec{U}_{0}(\theta)$ will denote the $MN\times1$ vectorization
of $U_{0}(\theta)$. We write $u_{0}(t,x)$ and $u(t,x;h)$ to denote
these functions on the domain $T\times X.$

The matrix $\nabla_{\theta}U_{0}(\theta)$ is the $MN\times k_{\theta}$
matrix vectorization of the gradient of the analytical solution with
respect to $\theta$ (also known as the sensitivities). The matrix
$\nabla_{\theta}U(h,\theta)$ will denote the numerically-computed
$MN\times k_{\theta}$ matrix for these sensitivity equations. The
vector $\vec{\epsilon}$ denotes the $MN\times1$ vector of realizations
of the Gaussian error terms.

We will perform our inverse problem for values of $h$ given by $h_{i}=\left(10\times2^{i-1}\right)^{-1},\ i=1,\ldots,7$.
For each value of $h$, we also use temporal step size $k=\lambda h$ for a value of $\lambda$ that
will satisfy the CFL condition, which is a necessary (but not sufficient)
condition for numerical stability. When describing a vector of step
sizes, we will let $\boldsymbol{h}=(h_{1},...,h_{7})^{T}$. In a slight
abuse of notation, we will write a vector of function values, $f(h)$,
at different step sizes as $f(\boldsymbol{h})=(f(h_{1}),...,f(h_{7}))^{T}.$

\subsection{Artificial Data Generation\label{subsec:Artificial-Data-Generation}}

We generate several artificial data sets from $U_{0}(\theta_{0})$
for this study. These data sets are created by adding Gaussian noise
to the analytical solution, written as the statistical model
\begin{equation}
y_{i,j}=u_{0}(t_{i},x_{j};\theta_{0})+\epsilon_{i,j},\ \epsilon_{i,j}\stackrel{i.i.d.}{\sim}\mathcal{N}(0,\eta^{2}),\ \ i=1,..,M,\ j=1,..,N\label{eq:data_form}
\end{equation}
for some ``true'' parameter value, $\theta_{0}\in\Theta.$ We will
generate data sets with different values of $N$ and $\eta$ for both
initial conditions, $\phi_{d}(x)$ and $\phi_{c}(x)$. Note that $M$
will be fixed at 6 for simplicity in all data sets considered: we
performed a similar analysis for data sets generated with larger values
of $M$ but found that the final results to be similar to the results
for increasing $N$ (results not shown). We choose $\theta_{0}=(0.3,0.5)^{T}$
for data sets where $\phi(x)=\phi_{d}(x)$ and $\theta_{0}=(0.3,0.4)^{T}$
for data sets where $\phi(x)=\phi_{c}(x)$. An example data set is
depicted against $u_{0}(t,x;\theta_{0})$ for $\eta^{2}=0.01,N=11$
in Figure ~\ref{fig:Artificial-data-plot}.

We will also perform the inverse problem for multiple data sets with
varying numbers of data points and data error levels. For $\phi(x)=\phi_{d}(x),$
we consider data sets for $N=\{11,30,51\}$ and $\eta=\{0,10^{-1},1.5\times10^{-1},2\times10^{-1},3\times10^{-1},5\times10^{-1},1\}.$
For $\phi(x)=\phi_{c}(x),$ we consider data sets for $N=\{11,31,51\}$
and $\eta=\{0,10^{-4},5\times10^{-4},10^{-3},10^{-2},5\times10^{-2},10^{-1},2\times10^{-1}\}.$
We will only show some results in the main text for ease of interpretation,
but all results for all data sets are provided in the supporting material.

\begin{figure}
\center{}\includegraphics[width=0.45\textwidth]{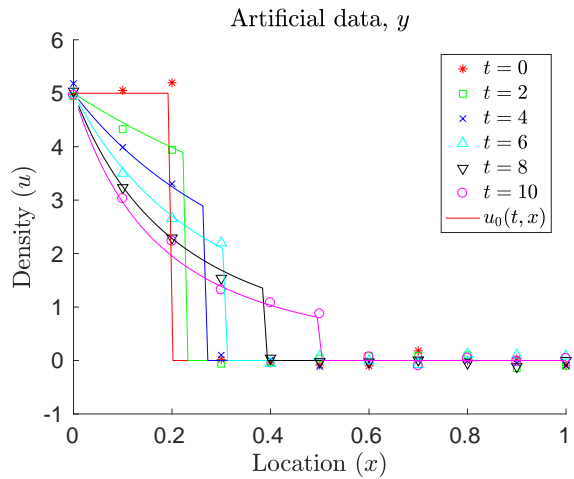}\caption{Artificial data from Equations (\ref{eq:adv_analtical}) and (\ref{eq:data_form})
for $\eta=0.1,N=11,M=6$, and $\phi(x)=\phi_{d}(x).$ The solid lines
denote the analytical solution given by Equation (\ref{eq:adv_analtical})
and the various markers denote the artificial data points. Red asterisks
denote $t=0,$ green squares denote $t=2,$ blue \emph{x}'s denote
$t=4,$ cyan triangles denote $t=6,$ black triangles denote $t=8,$
and magenta dots denote $t=10.$ \label{fig:Artificial-data-plot} }
\end{figure}

\subsection{Numerical Methods and Order of Convergence\label{subsec:Numerical-Methods}}

We will consider four commonly-used numerical schemes to approximate
the solution to Equation (\ref{eq:adv_eqn}). These four schemes are
the upwind, Lax-Wendroff, and Beam-Warming methods, as well as the
upwind method with flux limiters. The first three methods are discussed
and presented in the popular monograph by Leveque ~\citep{leveque_finite_2007},
and the final method is discussed in ~\citep[$\S 16.2$]{randall_j._leveque_conservation_1992}
and ~\citep{thackham_computational_2008}. 

A common practice in numerical analysis is to compute the order of
convergence for a numerical scheme. Guided by ~\citep{leveque_finite_2007},
we define the error for a numerical scheme as
\begin{align*}
E(h;\theta)=\|\vec{U}(h,\theta)-\vec{U}_{0}(\theta)\|_{1},
\end{align*}
where $\|\cdot\|_{1}$ denotes the 1-norm in $\text{\ensuremath{\mathbb{R}}}^{MN}$.
The upwind method is first order accurate when $u_{0}(t,x)$ is differentiable,
meaning that $E(h)\approx\mathcal{O}(h^{1})$ for $h$ sufficiently
small. The Lax-Wendroff and beam warming method are second order accurate
so that $E(h)\approx\mathcal{O}(h^{2})$ under similar assumptions. 

While these schemes are often referred to as first- or second-order
accurate, this is only true when the analytical solution, $u_{0}(t,x)$,
is continuous with respect to $x$. When $u_{0}(t,x)$ is discontinous,
the order for these schemes can be computed using the \emph{theory
of modified equations} (described in ~\citep[$\S$ 11]{randall_j._leveque_conservation_1992}).
This theory can show that the upwind method is of order 1/2, and the
Lax-Wendroff and Beam-Warming methods are of order 2/3 when $u_{0}(t,x)$
is discontinuous. This theory can also be used to demonstrate that
the upwind method will add numerical diffusion error when used to
approximate the solution to Equation (\ref{eq:adv_eqn}). Similarly,
the Lax-Wendroff and Beam-Warming methods will add numerical dispersion
error when used to approximate the solution to Equation (\ref{eq:adv_eqn}).
In both cases, the rates of diffusion or dispersion disappear as $h\rightarrow0$.
These numerical error patterns can be clearly seen in Figure ~\ref{fig:Numerical-solution-profiles}.

This information from the theory of modified equations has prompted
our use of the following definition for the order of convergence throughout
our study.
\begin{defn}
\label{Definition2-A-numerical}A numerical method has order of convergence
$p$ if, for $h$ small, 
\begin{align*}
u(t,x;h,\theta)\approx u_{0}(t,x;\theta)+h^{p}w(t,x;h)
\end{align*}
for some positive value, $p$, on all compact subsets of $\Theta_{comp}\subset\Theta$,
where $\|w(t,x;h)\|_{\infty}$ is uniformly bounded for small values
of $h$. Furthermore, for every $h$, $u(x,t;h,\theta):\Theta\rightarrow L^{1}(T\times X\rightarrow\mathbb{R})$
is continuous with respect to $\theta$. 
\end{defn}

Note that this definition is stronger than the standard definition
of numerical order of convergence, and it immediately implies $E(h,\theta)=\mathcal{O}(h^{p})$
so long as $T\times X$ is compact. For the first (second) order methods,
$w(t,x)$ represents numerical diffusion (dispersion) from the approximation
scheme. 

From the observation that $|u(t,x;h,\theta)-u_{0}(t,x;\theta)|\le Ch^{p}$
for all $t,x$ from the above definition,
 where ${\displaystyle C=\sup_{(t,x)\in T\times X}|w(t,x;h)|}$,
we will rewrite the above equation as
\begin{equation}
u(t,x;h,\theta)=u_{0}(t,x;\theta)+\mathcal{O}(h^{p})\label{eq:defn_2}
\end{equation}
for ease of notation. To estimate the order of convergence for a numerical
scheme throughout this study, we will find the best-fit line for the
natural log of the error, $\ln(E(\boldsymbol{h}))$, against $\ln(\boldsymbol{h})$.
The slope of this line will estimate $p$, which we will denote as
the \emph{numerical order of convergence}.

Flux limiters are a popular tool to aid numerical schemes for advection
equations with discontinuous solutions ~\citep{sweby_high_1984,thackham_computational_2008}.
When flux limiters are used, the spatial gradient at each computational
point is estimated at each time point. These estimations are used
to make the numerical scheme approximately second-order accurate near
smooth spatial points and first order accurate near points of discontinuity.
An upwind scheme with flux limiters thus prevents dispersive oscillations
from propagating near the discontinuity, and instead allows a small amount of
numerical diffusion in this region. In this study, we will use the
Van-Leer flux limiter ~\citep{randall_j._leveque_conservation_1992}.

In Table ~\ref{tab:order_estimate_phi_front}, we depict the calculated
values of $p$ for $\phi(x)=\phi_{d}(x)$ and $\theta=\theta_{0}$
and see that our calculated numerical order of convergence for the
upwind scheme is consistent with the theory (close to 1/2), but the
order for the Lax-Wendfroff and Beam-Warming schemes are smaller and
larger than expected, respectively ($p$ is calculated as 0.4737 for
Lax-Wendfroff and 0.7876 for Beam-Warming, when the theory suggests these
both should be 2/3)\footnote{Note that the Beam-Warming method uses one-sided derivative approximations
from the direction where information is coming from in its computations,
whereas the Lax-Wendroff method uses centered difference approximations.
The one-sided approximations are more accurate than centered difference
approximations for advection equations, which likely explains why
$p>2/3$ for the Beam-Warming method and why $p<2/3$ for the Lax-Wendroff
method. We depict simulations of both of these numerical schemes in
Figure ~\ref{fig:Numerical-solution-profiles} and indeed see that
the Lax-Wendfroff method is much more dispersive (\emph{i.e.}, less
accurate) than the Beam-Warming method.}. The upwind scheme with flux limiters has a calculated numerical
order of convergence of 0.9570. To the best of our knowledge, there
are no analytical results for the order of the upwind scheme with
flux limiters when calculating a discontinuous solution. We show in
Table S1 in the supporting material that the calculated numerical
orders of convergence for $\phi(x)=\phi_{c}(x)$ are consistent with
theory for continuous solutions. Here, we calculate the order of convergence
for the upwind with flux limiters scheme to be 0.9183.

\section{Asymptotic Properties of the Inverse Problem \label{sec:Numerical-Cost-Function}}

For a given data set, $Y,$ and (analytical) mathematical model, $U_{0}(\theta)$,
the ordinary least squares (OLS) cost function given by
\begin{equation}
J^{M,N}(\theta)=\dfrac{1}{MN}\sum_{i=1}^{M}\sum_{j=1}^{N}(y_{i,j}-u_{0}(t_{i},x_{j};\theta))^{2},\label{eq:OLS_cost_no_num}
\end{equation}
is a means to estimate the disparity between the data and model. In
the inverse problem framework, one may compute an estimate, $\hat{\theta}_{OLS}^{M,N}$,
of the true parameter vector, $\theta_{0}$, by finding
\begin{align*}
\hat{\theta}_{OLS}^{M,N}=\arg\min_{\theta\in\Theta_{ad}}J^{M,N}(\theta).
\end{align*}

In practice, we do not know $U_{0}(\theta)$ and must approximate it
with the numerical computation, $U(h,\theta)$. In this case, the \emph{numerical}
OLS cost function, 
\begin{equation}
J^{M,N}(h,\theta)=\dfrac{1}{MN}\sum_{i=1}^{M}\sum_{j=1}^{N}(y_{i,j}-u(t_{i},x_{j};h,\theta))^{2},\label{eq:OLS_cost_num}
\end{equation}
is a means to estimate the disparity between the data and numerical
computation. In this work, we will compute estimates, $\hat{\theta}_{OLS}^{M,N}(h)$,
of the true parameter vector, $\theta_{0}$, by finding
\begin{align*}
\hat{\theta}_{OLS}^{M,N}(h)=\arg\min_{\theta\in\Theta_{ad}}J^{M,N}(h,\theta),
\end{align*}
where $\Theta_{ad}\subset \Theta$ denotes the space of admissible parameter values. For the optimization in this study, we used an interior point algorithm
as implemented in MATLAB's \textbf{fmincon} function to find $\hat{\theta}_{OLS}^{M,N}(h)$. 

In the rest of this section, we will discuss asymptotic properties
of the inverse problem as the number of data points increases ($M,N\rightarrow\infty$)
and the step size decreases ($h\rightarrow0$). In Section ~\ref{subsec:Theory-on_thet_ols},
we discuss the asymptotic distribution of the estimator, $\hat{\theta}_{OLS}^{M,N}(h)$,
and in Section ~\ref{subsec:Convergence-of-J_ols}, we discuss the
convergence of the numerical cost function, $J^{M,N}(h,\hat{\theta}_{OLS}^{M,N}(h)).$

\subsection{Theory of $\hat{\theta}_{OLS}^{M,N}(h)$\label{subsec:Theory-on_thet_ols}}

The asymptotic properties of $\theta_{OLS}^{M,N}$ have been widely
discussed and are provided in Theorem 2.1 from ~\citep{seber_nonlinear_1988}
(which is stated in ~\ref{sec:Theory-on-theta_ols_h} for
convenience). In ~\citep{banks_statistical_1990}, it is further shown
that $\hat{\theta}_{OLS}^{M,N}(h)$ is a consistent estimator, meaning
that $\hat{\theta}_{OLS}^{M,N}(h)\rightarrow\theta_{0}$ almost surely
as $M,N\rightarrow\infty$ and $h\rightarrow0$. This proof requires
$\Theta_{ad}$ to be a compact subset of $\Theta$; accordingly, we
have chosen $\Theta_{ad}=[0,10]\times[0,10]$. This proof also requires
the following reasonable assumptions.
\begin{enumerate}
\item[(A1)] The finite measures $\chi$ and $\nu$ exist on $X$ and $T$
such that 
\begin{align*}
\dfrac{1}{MN}\sum_{i=1}^{M}\sum_{j=1}^{N}v(t_{i},x_{j})\rightarrow\int_{X}\int_{T}v(t,x;\theta)d\nu(t)d\chi(x)
\end{align*}
for any $v\in L^{1}(T\times X\rightarrow\mathbb{R})$ as $M,N\rightarrow\infty$. 
\item[(A2)]  The functional
\begin{align*}
J^{*}(\theta)=\int_{X}\int_{T}(u_{0}(t,x;\theta_{0})-u_{0}(t,x;\theta))^{2}d\nu(t)d\chi(x)
\end{align*}
has a unique minimizer in $\Theta_{ad}$ at $\theta_{0}.$
\end{enumerate}
The theorem for $\theta_{OLS}^{M,N}$ in ~\citep{seber_nonlinear_1988}
does not account for numerical errors in the model solution while
the theory for $\hat{\theta}_{OLS}^{M,N}(h)$ in ~\citep{banks_statistical_1990}
does not consider the implications of the numerical order of solution
convergence. 

We now state our main theoretical result on the behavior of $\theta_{OLS}^{M,N}(h)$
as numerical accuracy increases. The following corollary extends the
above theory to account for the fact that the solution to the PDE
model is being approximated with an order $p$ scheme.
\begin{cor}
\label{cor:2}Consider a numerical scheme for a differential equation
that is order $p$ accurate for $u_{0}(t,x)$ and $\nabla_{\theta}u_{0}(t,x;\theta)$.
Under the assumptions (A1) and (A2), we have the asymptotic distribution
for $\theta_{OLS}^{M,N}(h)$ as $M,N\rightarrow\infty$ and $h\rightarrow0$
given by
\begin{align*}
\hat{\theta}_{OLS}^{M,N}(h)\sim\mathcal{N}(\theta_{0},\eta^{2}V_{h}).
\end{align*}
The entries of $V_{h}$ are $\mathcal{O}(h^{p})$ convergent to the
entries of $V=(\nabla_{\theta}U_{0}(\theta_{0})^{T}\nabla_{\theta}U_{0}(\theta_{0}))^{-1}$,
which is the covariance matrix in the absence of numerical error.
\end{cor}

\begin{proof}
When $\theta$ is near $\theta_{0},$ we Taylor expand to see 
\begin{align*}
\vec{U}(h,\theta) & \approx\vec{U}(h,\theta_{0})+\nabla_{\theta}U(h,\theta_{0})[\theta-\theta_{0}],
\end{align*}
and then use our assumptions on the numerical orders of convergence
to find
\begin{align*}
\vec{U}(h,\theta)\approx\vec{U}_{0}(\theta_{0})+\mathcal{O}(h^{p})+[\nabla_{\theta}U_{0}(\theta_{0})+\mathcal{O}(h^{p})][\theta-\theta_{0}].
\end{align*}
The numerical cost function then takes the form
\begin{align*}
J_{OLS}(h,\theta) & =\|\vec{Y}-\vec{U}(h,\theta)\|^{2}\\
 & \approx\left\Vert \vec{Y}-\vec{U}_{0}(\theta_{0})-\mathcal{O}(h^{p})-[\nabla_{\theta}U_{0}(\theta_{0})+\mathcal{O}(h^{p})][\theta-\theta_{0}]\right\Vert ^{2}.
\end{align*}
The first $\mathcal{O}(h^{p})$ term is independent of $\theta$ on
$\Theta_{ad}$, so minimizing $J_{OLS}(h,\theta)$ is equivalent to
minimizing 
\begin{align*}
\|\vec{Z}-X_{h}\beta\|^{2},
\end{align*}
where $\vec{Z}=\vec{Y}-\vec{U}_{0}(\theta_{0})$, $X_{h}=\nabla_{\theta}U_{0}(t,x;\theta_{0})+\mathcal{O}(h^{p})$,
and $\beta=\theta-\theta_{0}$. The above has the minimizer 
\begin{equation}
\hat{\beta}=\left(X_{h}^{T}X_{h}\right)^{-1}X_{h}^{T}\vec{Z},\label{eq:q_ols_eps-1}
\end{equation}
which is normally distributed because it is a linear combination of
normal random variables. Assumptions $(A1)$ and $(A2)$ ensure that
$\hat{\theta}_{OLS}^{M,N}(h)$ is consistent for estimating $\theta_{0}$
as $h\rightarrow0$ and $M,N\rightarrow\infty$ ~\citep{banks_statistical_1990}.
Once $\hat{\theta}_{OLS}^{M,N}(h)$ is close to $\theta_{0},$ we
have that
\begin{align*}
\hat{\theta}_{OLS}^{M,N}(h)\approx\mathcal{N}(\theta_{0},\eta^{2}(X_{h}^{T}X_{h})^{-1}),
\end{align*}
where the mean and covariance can be calculated directly from their
definitions.

Determining the convergence (in any matrix norm) of $V_{h}=(X_{h}^{T}X_{h})^{-1}$
to the inverse of $\nabla_{\theta}U_{0}(\theta_{0})^{T}\nabla_{\theta}U_{0}(\theta_{0})$
is a difficult problem. However, by using a result from the analysis
of numerical algorithms ~\citep{higham_accuracy_1996}, we can draw
some conclusions about the individual entries of $V_{h}$. Consider
the $(i,j)th$ entry of $X_{h}^{T}X_{h}:$
\begin{align*}
|[X_{h}^{T}X_{h}]_{i,j}| & =|(\nabla_{\theta}U_{0}(t_{i},x_{j};\theta_{0})+\mathcal{O}(h^{p}))^{T})(\nabla_{\theta}U_{0}(t_{i},x_{j};\theta_{0})+\mathcal{O}(h^{p}))|\\
 & =|\nabla_{\theta}U_{0}(t_{i},x_{j};\theta_{0})^{T}\nabla_{\theta}U_{0}(t_{i},x_{j};\theta_{0})+\mathcal{O}(h^{p})\nabla_{\theta}U_{0}(t_{i},x_{j};\theta_{0}))+\mathcal{O}(h^{2p})|,
\end{align*}
 meaning that this entry converges to its corresponding entry of $\nabla_{\theta}U_{0}(\theta_{0})^{T}\nabla_{\theta}U_{0}(\theta_{0})$
with an order of convergence $p.$ Then, using results
from ~\citep[$\S$ 13.1]{higham_accuracy_1996}, we can show 
\begin{align*}
 & \left|\left[(\nabla_{\theta}U_{0}(\theta_{0})^{T}\nabla_{\theta}U_{0}(\theta_{0}))^{-1}-(X_{h}^{T}X_{h})^{-1}\right]{}_{i,j}\right| \le \\
& \mathcal{O}(h^{p})\left|\left[(\nabla_{\theta}U_{0}(\theta_{0})^{T}\nabla_{\theta}U_{0}(\theta_{0}))^{-1}\right]_{i,j}\right|\left|\left[ \nabla_{\theta}U_{0}(\theta_{0})^{T}\nabla_{\theta}U_{0}(\theta_{0})\right]_{i,j}\right|\left|\left[(\nabla_{\theta}U_{0}(\theta_{0})^{T}\nabla_{\theta}U_{0}(\theta_{0}))^{-1}\right]_{i,j}\right|\\ & +\mathcal{O}(h^{2p}).
\end{align*}
Thus, each entry of $V_{h}$ will will converge to
its corresponding entry of $V$ as $\mathcal{O}(h^{p})$.
\end{proof}

\subsection{Convergence of $J^{M,N}(h,\theta)$\label{subsec:Convergence-of-J_ols}}

The least squares cost function from Equation (\ref{eq:OLS_cost_num})
is widely used for inverse problems ~\citep{banks_mathematical_2009}.
In this section, we discuss the asymptotic properties of this function
as $h\rightarrow0$ and $M,N\rightarrow\infty$ to elucidate our results
in future sections.

Observe that by combining Equations (\ref{eq:data_form}) and (\ref{eq:OLS_cost_num}),
the cost function can be rewritten as 
\begin{align}
J^{M,N}(h,\theta) & =\dfrac{1}{MN}\sum_{i,j=1}^{M,N}[u_{0}(t_{i},x_{j};\theta_{0})+\epsilon_{i,j}-u(t_{i},x_{j};h,\theta)+u_{0}(t_{i},x_{j};\theta)-u_{0}(t_{i},x_{j};\theta)]^{2}\label{eq:cost_break_down}\\
 & =A+B+C+D+E+F,\nonumber 
\end{align}
where 
\begin{align}
A & =\dfrac{1}{MN}\sum_{i,j=1}^{M,N}\mathcal{\epsilon}_{i,j}^{2}\nonumber \\
B & =\dfrac{1}{MN}\sum_{i,j=1}^{M,N}[u_{0}(t_{i},x_{j};\theta_{0})-u_{0}(t_{i},x_{j};\theta)]^{2}\nonumber \\
C & =\dfrac{1}{MN}\sum_{i,j=1}^{M,N}[u_{0}(t_{i},x_{j};\theta)-u(t_{i},x_{j};h,\theta)]^{2}\nonumber \\
D & =\dfrac{2}{MN}\sum_{i,j=1}^{M,N}\mathcal{\epsilon}_{i,j}(u_{0}(t_{i},x_{j};\theta_{0})-u_{0}(t_{i},x_{j};\theta))\nonumber \\
E & =\dfrac{2}{MN}\sum_{i,j=1}^{M,N}\mathcal{\epsilon}_{i,j}(u_{0}(t_{i},x_{j};\theta)-u(t_{i},x_{j};h,\theta))\nonumber \\
F & =\dfrac{2}{MN}\sum_{i,j=1}^{M,N}[(u_{0}(t_{i},x_{j};\theta)-u(t_{i},x_{j};h,\theta))(u_{0}(t_{i},x_{j};\theta_{0})-u_{0}(t_{i},x_{j};\theta))]\label{eq:J_components}
\end{align}
We thus observe that the numerical cost function can be broken down
into six separate terms, each of which converges. The two following lemmas discuss
 the asymptotic limits and orders of convergence for terms $A$ through $F$
 as data increases and as numerical accuracy increases.
\begin{lem}
\label{lem:num_conv}If the numerical method is order $p$ accurate
for $u_{0}(t,x)$ and $\nabla_{\theta}u_{0}(t,x)$, then the terms
A-F from Equation (\ref{eq:J_components}) will behave as follows
as $h\rightarrow0$:

\begin{align*}
A\approx\mathcal{O}(1),\ B\approx\mathcal{O}(h^{p}),\ C\approx\mathcal{O}(h^{2p}),\ D\approx\mathcal{O}(h^{p/2}),\ E\approx\mathcal{O}(h^{p}),\ F\approx\mathcal{O}(h^{3p/2}).
\end{align*}
\end{lem}

\begin{proof}
See ~\ref{subsec:Limits-ash}.
\end{proof}
\begin{lem}
\label{lem:data_conv} If the numerical method is order $p$ accurate
for $u_{0}(t,x)$ and $\nabla_{\theta}u_{0}(t,x)$, then the terms
A-F from Equation (\ref{eq:J_components}) will behave as follows
as $M,N\rightarrow0$:

$A$ will converge to 0 with order $\mathcal{O}(1/\sqrt{MN})$. $B$
will converge to the functional $J^{*}(\theta)$ with order $\mathcal{O}(1/(MN)).$
C is independent of $M$ and $N$. D will converge to 0 with order
$\mathcal{O}(1/(\sqrt{MN}))$. E will Converge to 0 with order $\mathcal{O}(1/(\sqrt{MN}))$.
$F$ will converge to an $\mathcal{O}(h^{p})$ term with order $\mathcal{O}(1/(MN))$.
\end{lem}

\begin{proof}
See ~\ref{subsec:Limits-asMN}.
\end{proof}
These two lemmas are summarized in Table ~\ref{tab:Asymptotic-limits-for}. 

\begin{table}
\centering{}%
\begin{tabular}{|c|c|c|}
\hline 
\multirow{2}{*}{} & \multirow{2}{*}{\textbf{Asymptotic Properties ($h\rightarrow0$)}} & \multirow{2}{*}{\textbf{Asymptotic Properties (}$M,N\rightarrow\infty$\textbf{)}}\tabularnewline
 &  & \tabularnewline
\hline 
\multirow{2}{*}{A} & \multirow{2}{*}{$\mathcal{O}(1)$} & \multirow{2}{*}{Converges to $\eta^{2}$ with order $\mathcal{O}(1/\sqrt{MN})$}\tabularnewline
 &  & \tabularnewline
\hline 
\multirow{2}{*}{B} & \multirow{2}{*}{$\mathcal{O}(h^{p})$} & \multirow{2}{*}{Converges to $J^{*}(\theta)$ with order $\mathcal{O}(1/(MN))$}\tabularnewline
 &  & \tabularnewline
\hline 
\multirow{2}{*}{C} & \multirow{2}{*}{$\mathcal{O}(h^{2p})$} & \multirow{2}{*}{independent of $M,N$}\tabularnewline
 &  & \tabularnewline
\hline 
\multirow{2}{*}{D} & \multirow{2}{*}{$\mathcal{O}(h^{p/2})$} & \multirow{2}{*}{Converges to 0 with order $\mathcal{O}(1/(\sqrt{MN}))$}\tabularnewline
 &  & \tabularnewline
\hline 
\multirow{2}{*}{E} & \multirow{2}{*}{$\mathcal{O}(h^{p})$} & \multirow{2}{*}{Converges to 0 with order $\mathcal{O}(1/(\sqrt{MN}))$}\tabularnewline
 &  & \tabularnewline
\hline 
\multirow{2}{*}{F} & \multirow{2}{*}{$\mathcal{O}(h^{3p/2})$} & \multirow{2}{*}{Converges to an $\mathcal{O}(h^{p})$ term with order $\mathcal{O}(1/(MN))$}\tabularnewline
 &  & \tabularnewline
\hline 
\end{tabular}\caption{Asymptotic limits for the six terms comprising the numerical cost
function given by equation (\ref{eq:OLS_cost_num}) as numerical accuracy increases ($h\rightarrow 0$) and as the number of data points increases ($M,N\rightarrow \infty$). \label{tab:Asymptotic-limits-for}}
\end{table}

\section{Inverse Problem Results \label{sec:Inverse-Problem-Results}}

In this section, we present and discuss the numerical results for our inverse problems
 as $h\rightarrow0$ and $N\rightarrow\infty.$
\footnote{We do not present the results for $M\rightarrow\infty$ as they are
identical to those presented here for $N\to\infty$.} In Section ~\ref{subsec:Numerical-Simulation-Profiles}, we discuss
the profiles of the numerical simulations that led to these results.
In Sections ~\ref{subsec:Behavior-of-Numerical} and ~\ref{subsec:Behavior-of-the-OLS-estimator},
we discuss the asymptotic behavior of $J^{M,N}(h,\hat{\theta}_{OLS}^{M,N}(h))$
and $\hat{\theta}_{OLS}^{M,N}(h)$, respectively.

\subsection{Numerical Simulation Profiles\label{subsec:Numerical-Simulation-Profiles}}

In Figure ~\ref{fig:Numerical-solution-profiles}, we depict a selection
of best-fit plots of $u(t,x;h,\hat{\theta}_{OLS}^{M,N}(h))$ against
their corresponding artificial data sets (for all four schemes) for
$\phi(x)=\phi_{d}(x).$ As expected, the first order upwind scheme
is diffusive, and the second order methods are dispersive. The Lax-Wendroff
method is excessively dispersive, as it displays many small oscillations
but still fits the general trend of the data. The Beam-Warming method
yields more accurate profile simulations than the Lax-Wendroff method, but
it does have a negative portion just after the front. The upwind method
with flux limiters provides the most realistic profile, as it maintains
a sharp front with a nonnegative profile.

\begin{sidewaysfigure}
\vspace{15cm}\includegraphics[scale=0.5]{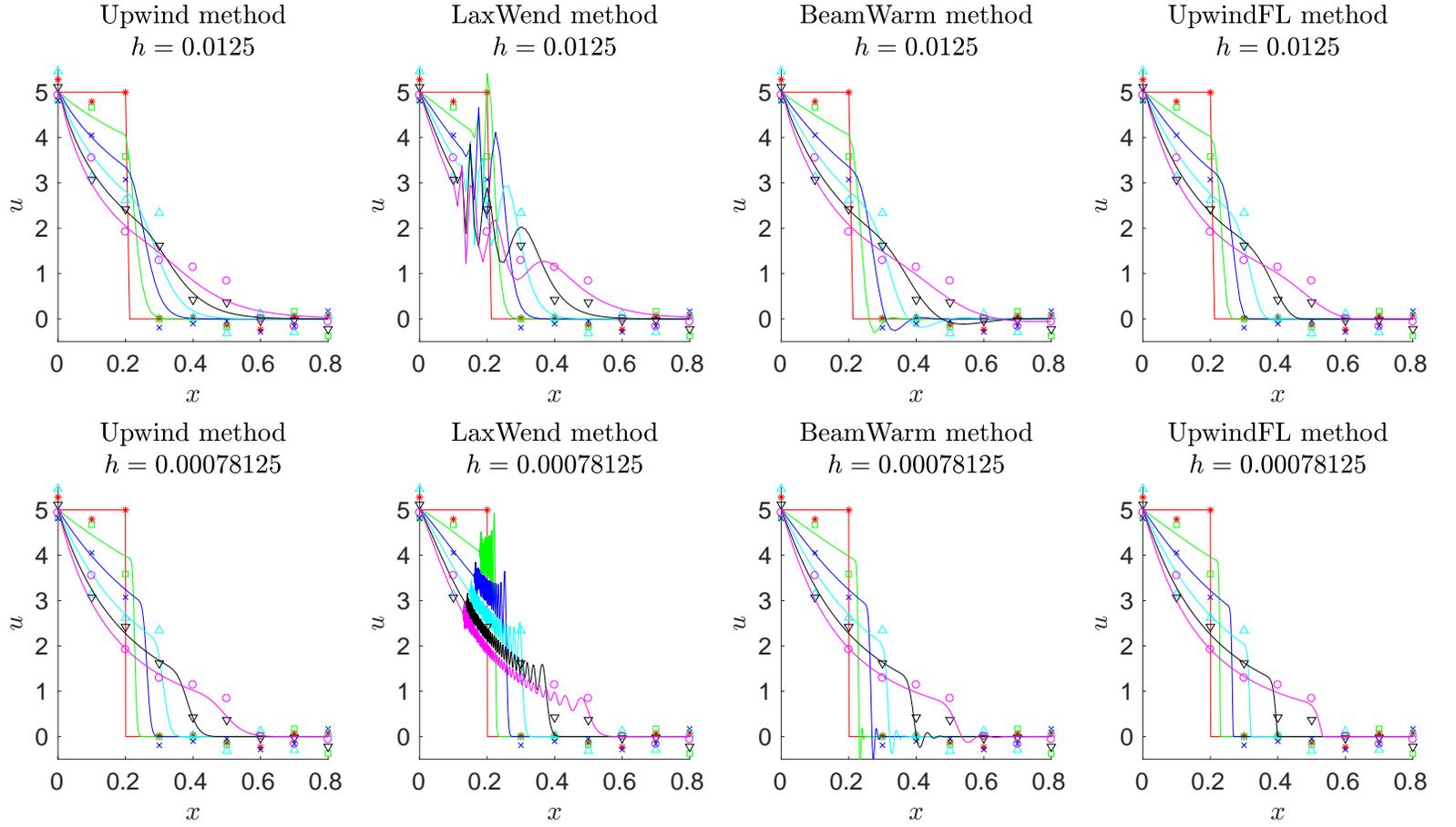}\caption{Numerical solution profiles (solid lines) plotted against artificial
data (dots) for the four schemes considered when $\phi(x)=\phi_{d}(x)$
for two different step sizes. Red asterisks denote $t=0,$ green squares
denote $t=2,$ blue \emph{x}'s denote $t=4,$ cyan triangles denote
$t=6,$ black triangles denote $t=8,$ and magenta dots denote $t=10.$
The solid curves denote $u(t,x;h,\hat{\theta}_{OLS}^{M,N}(h))$ at
these time points. In the titles, ``LaxWend'' corresponds to the
Lax-Wendroff method, ``BeamWarm'' corresponds to the Beam-Warming
Method, and ``UpwindFL'' corresponds to the Upwind method with flux
limiters. \label{fig:Numerical-solution-profiles}}
\end{sidewaysfigure}

\subsection{Behavior of Numerical Cost Function \label{subsec:Behavior-of-Numerical}}

In Figure ~\ref{fig:Plots-of-J_main_text}, we depict log-log plots
of $J^{M,N}(\boldsymbol{h},\hat{\theta}_{OLS}^{M,N}(\boldsymbol{h}))$
against $\boldsymbol{h}$ for an initial condition of $\phi(x)=\phi_{d}(x)$.
Here, we observe that the cost function converges to $\eta^{2}$ as
$h\rightarrow0$, which is consistent with the theory from Table ~\ref{tab:Asymptotic-limits-for}.
This observation suggests that numerical error dominates over experimental
error until $J^{M,N}(h,\hat{\theta}_{OLS}^{MN}(h))$ reaches $\eta^{2}$,
at which point experimental error becomes the dominant term in $J^{M,N}(h,\hat{\theta}_{OLS}^{MN}(h))$.
We thus suggest that if $J^{M,N}(h,\hat{\theta}_{OLS}^{M,N}(h))$
decreases with $h$ , then one can further decrease the value of the cost
function with continued grid refinement. We depict the log-log plots
of $J^{M,N}(\boldsymbol{h},\hat{\theta}_{OLS}^{M,N}(\boldsymbol{h}))$
for all data sets considered in the supporting material in Figure
S2 for $\phi(x)=\phi_{c}(x)$ and in Figure S12 for $\phi(x)=\phi_{d}(x)$;
these figures support the observations that $J^{M,N}(h,\hat{\theta}_{OLS}^{M,N}(h))$
converges to $\eta^{2}$ as $h\rightarrow0$. We also observe, as
expected from Table ~\ref{tab:Asymptotic-limits-for}, that $J^{M,N}(h,\hat{\theta}_{OLS}^{M,N}(h))$
gets closer to $\eta^{2}$ as $M,N\rightarrow\infty.$ If one is concerned
with accurately estimating $\eta^{2},$ then they can use $J^{M,N}(h,\hat{\theta}_{OLS}^{M,N}(h))$
with more certainty for large values of $M,N$.

\begin{figure}
\centering{}\-\-\-\includegraphics[width=0.99\textwidth]{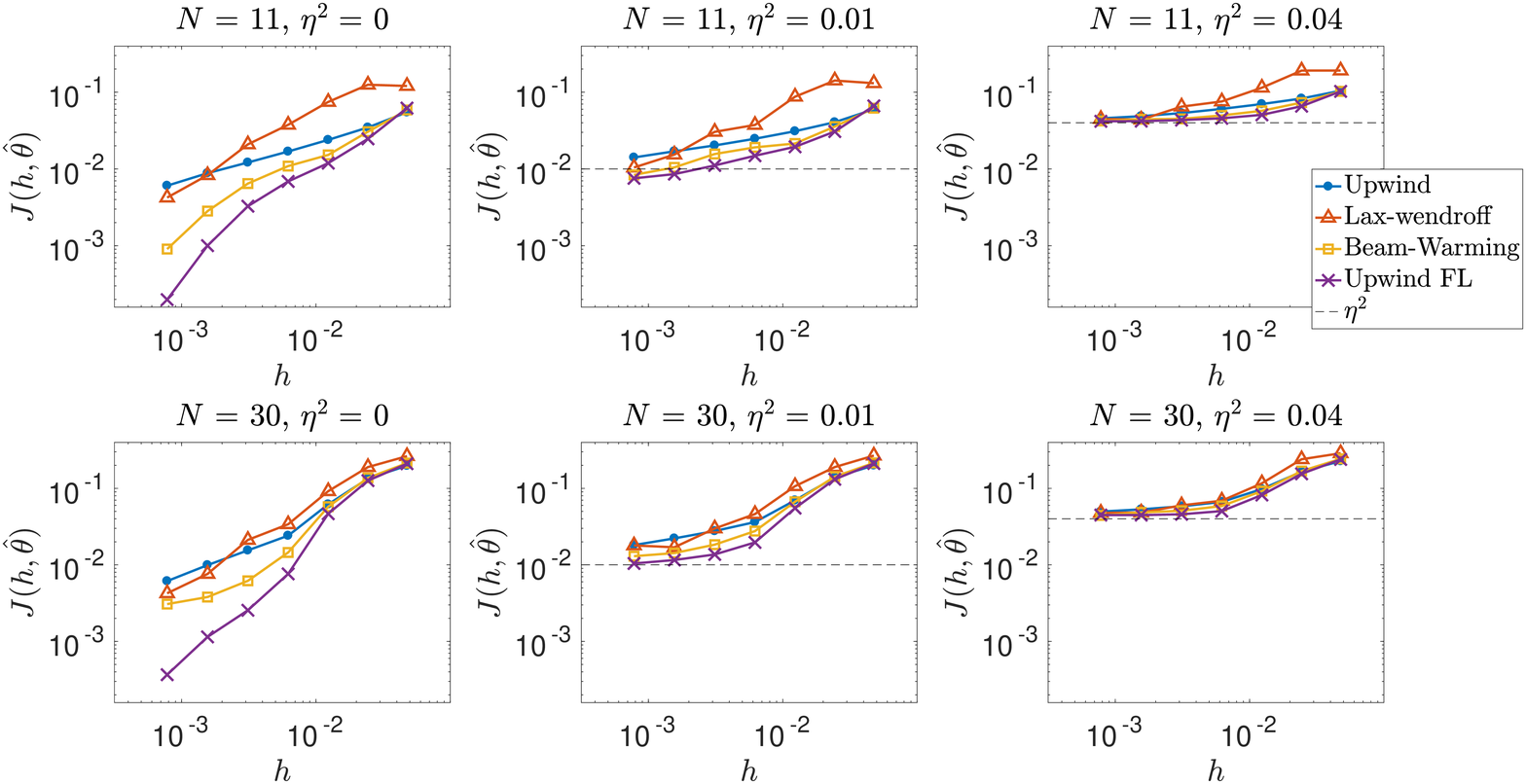}\caption[Plots of the numerical cost function for $\phi(x)=\phi_{1}(x).$]{Plots of $J_{OLS}^{M,N}(h,\hat{\theta}_{OLS}^{M,N}(\boldsymbol{h}))$
for the four schemes considered with $\phi(x)=\phi_{d}(x)$. We depict
the results for $N=11$ or 30 and $\eta^{2}=0,0.01,\text{ or }0.04$\-.
\label{fig:Plots-of-J_main_text} }
\end{figure}

In Figure ~\ref{fig:Plots-of-J_main_text}, we observe that $J^{M,N}(h,\hat{\theta}_{OLS}^{M,N}(h))$
appears to converge differently based on the numerical method used.
To confirm this, we estimate the order of convergence of the numerical
cost function by fitting the best-fit line between $\log\left(J^{M,N}(\vec{h},\hat{\theta}_{OLS}^{M,N}(\boldsymbol{h}))\right)$
and $\log(\boldsymbol{h})$. The slope of this line denotes the order
of convergence for the numerical cost function, and we denote this
calculation as\footnote{Note that we use values of $h$ where $\log\left(J^{M,N}(h,\hat{\theta}_{OLS}^{M,N}(h))\right)$
has not yet converged to $\eta^{2}$ when computing $p_{J}$ (for
example, for $\phi(x)=\phi_{d}(x),N=30,\eta^{2}=0.04,$ we use the
four coarsest points to compute the order for the upwind scheme with
flux limiters).} $p_{J}$. We present some results for $\phi(x)=\phi_{d}(x)$ in Table
~\ref{tab:order_estimate_phi_front}. We observe that $p_{J}$ is about
the same as $p$ for the upwind and Beam-Warming schemes and double
the value of $p$ for the Lax-Wendroff Scheme when $\eta^{2}=0$.
As $\eta^{2}$ increases, this value decreases. There is no apparent
pattern between $p_{J}$ and $p$ for the upwind scheme with flux
limiters. In the supporting material, we depict the values for $p_{J}$
for all data sets considered for $\phi(x)=\phi_{c}(x)$ in Table S1
and for $\phi(x)=\phi_{d}(x)$ in Table S3. For the continuous solutions
when $\phi(x)=\phi_{c}(x),$ we observe that $p_{J}$ is often double
the value of $p$. The order tends to decrease as $\eta$ increases
for both continuous and discontinuous solutions, eventually reaching
zero when experimental error dominates numerical error for all values
in \textbf{$\boldsymbol{h}$}.

\begin{table}
\centering%
\begin{tabular}{|c|c|c|c|c|c|c|c|c|}
\hline 
\multirow{2}{*}{Numerical Method} & \multirow{2}{*}{$p$} &  & \multicolumn{3}{c|}{$p_{J}$} & \multicolumn{3}{c|}{$p_{\theta}$}\tabularnewline
\cline{3-9} 
 &  & \backslashbox{$N$}{$\eta^2$} & $0$ & $4\times10^{-2}$ & $1$ & $0$ & $4\times10^{-2}$ & $1$\tabularnewline
\hline 
\multirow{2}{*}{Upwind} & \multirow{2}{*}{0.5839} & $11$ & 0.517  & 0.208  & -0.002  & 0.360  & 0.446  & 0.406 \tabularnewline
\cline{3-9} 
 &  & 30 & 0.612  & 0.226  & 0.040  & 0.515  & 0.447  & 0.608 \tabularnewline
\hline 
\multirow{2}{*}{Lax-Wendroff} & \multirow{2}{*}{0.4737} & $11$ & 0.966  & 0.490  & -0.011  & 0.463  & 0.680  & 0.355 \tabularnewline
\cline{3-9} 
 &  & 30 & 0.878 & 0.387  & 0.062  & 1.023  & 0.523  & 0.213 \tabularnewline
\hline 
\multirow{2}{*}{Beam-Warming} & \multirow{2}{*}{.7876} & $11$ & 0.785  & 0.367  & 0.000  & 0.769  & 0.525  & -0.077 \tabularnewline
\cline{3-9} 
 &  & 30 & 0.987  & 0.441  & 0.040  & 0.380  & 0.518  & 0.303 \tabularnewline
\hline 
\multirow{2}{*}{Upwind FL} & \multirow{2}{*}{.9570} & 11 & 1.285  & 0.409  & -0.020  & 0.582  & 0.510  & -0.199 \tabularnewline
\cline{3-9} 
 &  & 30 & 1.338  & 0.505  & 0.037  & 0.189  & 0.523  & -0.538 \tabularnewline
\hline 
\end{tabular}\caption{Table of computed numerical and statistical orders of convergence
for each data set when computed with different numerical methods for
$\phi(x)=\phi_{d}(x)$. The variable $p$ denotes the computed order
of numerical accuracy, $p_{J}$ denotes the computed order of convergence
for the numerical cost function, and $p_{\theta}$ denotes the computed
order of convergence for $\|\hat{\theta}_{OLS}^{M,N}(h)-\theta_{0}\|_{2}$.\label{tab:order_estimate_phi_front}}
\end{table}

It is at first puzzling that $p_{J}\approx2p$ for the Lax-Wendroff
method when $\phi(x)=\phi_{d}(x)$, yet $p_{J}\approx p$ for the
upwind and Beam-Warming methods. This behavior can be explained, however,
by looking at the terms A-F from Equation (\ref{eq:J_components})
that result from these different computations. In Figures S14-S17
in the supporting material, we depict the components A through F against
$J_{OLS}^{M,N}(\boldsymbol{h},\hat{\theta}_{OLs}^{M,N}(\boldsymbol{h}))$
for all data sets and for all numerical schemes used. For the upwind
scheme, the $\mathcal{O}(h^{p})$ term $B$ is on the same order of
magnitude as $J_{OLS}^{M,N}(h,\hat{\theta}_{OLs}^{M,N}(h))$, which
causes $p_{J}\approx p$. For the Lax-Wendroff scheme, the $\mathcal{O}(h^{2p})$
term $C$ tends to dominate the numerical cost function as $h$ decreases.
This different behavior of terms A through F for different numerical
methods is a likely explanation for the different $p_{J}/p$ ratios
computed for our numerical methods. We depict these plots of A-F for
all schemes considered in the supporting material in Figures S4-S7
for $\phi(x)=\phi_{c}(x)$.

\subsection{Behavior of the Numerical OLS Estimator\label{subsec:Behavior-of-the-OLS-estimator}}

In Figure ~\ref{fig:Plots-of-q_main_text}, we depict plots of $\|\hat{\theta}_{OLS}^{M,N}(\boldsymbol{h})-\theta_{0}\|_{2}$
against $\boldsymbol{h}$ for $\phi(x)=\phi_{d}(x)$. The ``Upwind
auto'' estimates modify cost function computation and will be discussed
later in Section ~\ref{subsec:Residual-Analysis} with our residual
analysis. In this figure, we observe that it is hard to predict which
scheme will estimate $\theta_{0}$ best. For example, the Beam-Warming
and upwind with flux limiter schemes tend to estimate $\theta_{0}$
best out of all methods considered. The Lax-Wendroff method also provides
the best estimate of $\theta_{0}$ in some cases, however, but its
accuracy is unpredictable. Recall from Figure ~\ref{fig:Numerical-solution-profiles}
that the Lax-Wendroff method computes very dispersive $u(t,x;h,\theta)$
profiles. These dispersive oscillations are a likely explanation for
the somewhat unpredictable $\hat{\theta}_{OLS}^{M,N}(h)$ estimates
for this method. It is possible that numerical simulations that are
computed with parameter vectors close to $\theta_{0}$ cause oscillations
that prevent the numerical approximation from matching the data closely,
whereas numerical simulations that are computed at vectors farther
from $\theta_{0}$ cause oscillations that help the numerical approximation
match the given data points. We depict plots of $\|\hat{\theta}_{OLS}^{M,N}(\boldsymbol{h})-\theta_{0}\|_{2}$
for all data sets considered in the supporting material in Figure
S3 for $\phi(x)=\phi_{c}(x)$ and in Figure S13 for $\phi(x)=\phi_{d}(x).$
These figures show that the Beam-Warming and Lax-Wendroff schemes
often do best for $\phi(x)=\phi_{c}(x)$ and confirm that the best
method is hard to declare for $\phi(x)=\phi_{d}(x).$ 

\begin{figure}
\centering{}\includegraphics[width=0.99\textwidth]{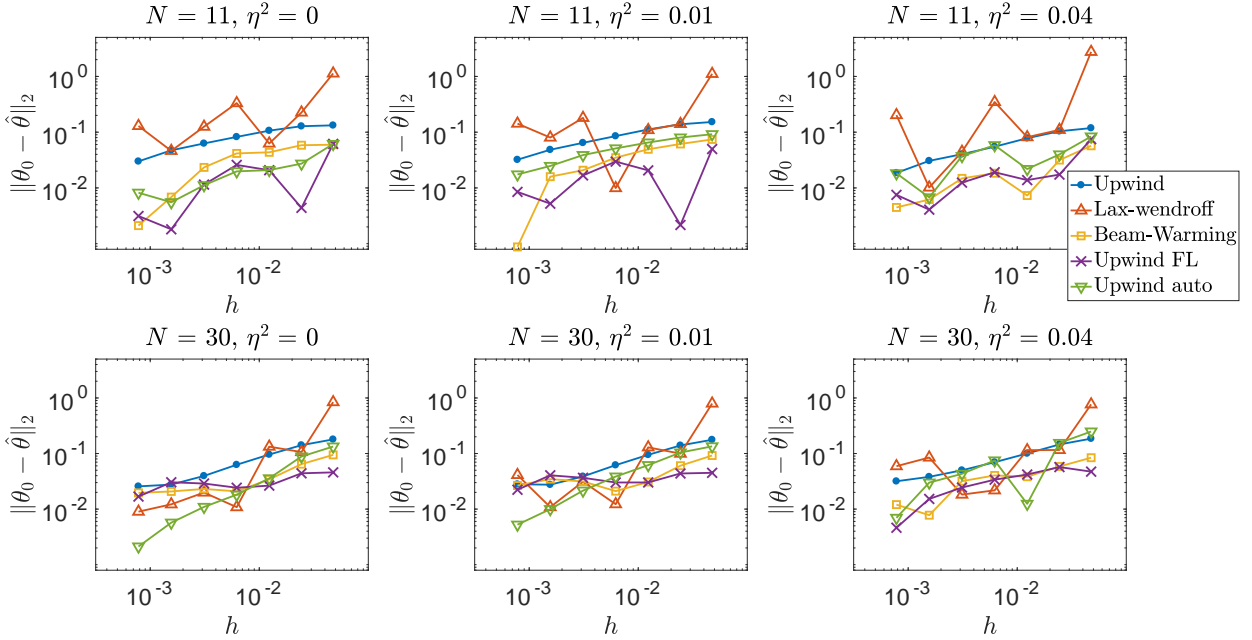}\caption[Plots of the numerical cost function for $\phi(x)=\phi_{1}(x).$]{Plots of $\|\theta_{0}-\hat{\theta}_{OLS}^{M,N}(\boldsymbol{h})\|_{2}$
for the four schemes considered with $\phi(x)=\phi_{d}(x)$. We depict
the results for $N=11$ or 30 and $\eta^{2}=0,0.01,$ or 0.04. \label{fig:Plots-of-q_main_text} }
\end{figure}

We depict a representative selection of computed orders of convergence
for $\|\hat{\theta}_{OLS}^{M,N}(\boldsymbol{h})-\theta_{0}\|_{2}$
(denoted as $p_{\theta}$) in Table ~\ref{tab:order_estimate_phi_front}.
All results for $\phi(x)=\phi_{d}(x)$ are included in Table S4 in
the supporting material. We observe that $p_{\theta}\approx p$ for
the upwind and Beam-Warming schemes and $p_{\theta}\approx p$ or
$p_{\theta}\approx2p$ for the Lax-Wendroff Scheme when $\eta^{2}=0$.
There is no apparent pattern between $p_{\theta}$ and $p$ for the
upwind scheme with flux limiters, although often $p_{\theta}\approx0$
for this method. Recall from Corollary ~\ref{cor:2} that we expect
$\hat{\theta}_{OLS}^{M,N}(h)$ to asymptotically behave as a random
variable with mean $\theta_{0}$ and a variance that converges as
$\mathcal{O}(h^{p}).$ This may explain why many estimates are converging
with $p_{\theta}\approx p$: they converge as their variance. It is not clear
why $p_{\theta}\approx2p$ for some results with the Lax-Wendroff
method. In the supporting material, we depict the values for $p_{\theta}$
for all data sets considered for $\phi(x)=\phi_{c}(x)$ in Table S2 and see that
 $p_\theta\approx p$ for all numerical methods.

\section{Residual Analysis and Confidence Intervals \label{sec:Residual-Analysis}}

In Figure ~\ref{fig:ols_resid_2}, We depict the residuals for the
upwind method, along with $u(t_{i},x_{j};h,\hat{\theta}_{OLS}^{M,N}(h))$,
and observe that local correlations in residual values arise near
the point of discontinuity. Accordingly, in this section, we will
explore how using an autocorrelative statistical model can improve
uncertainty quantification for our inverse problem when using the
first-order upwind method. In Section ~\ref{subsec:Residual-Analysis},
we will use residual analysis to derive this statistical model. We
will demonstrate how this statistical model can improve confidence
interval computation in Section ~\ref{subsec:Confidence-Interval-Computation}.

\subsection{Residual Analysis\label{subsec:Residual-Analysis}}

The statistical model describes how the underlying mathematical model
is observed through experimental data. Residuals can be used to help
practitioners ascertain the underlying statistical model of their
data ~\citep{banks_extension_2012}. If numerical error is prevalent
in a practitioner's computation, then it is interesting to consider
how numerical error propagates in residual computation. Here we will
develop an autocorrelative statistical model to describe how numerical
error propagates in the inverse problem when using the upwind method
for numerical computation when $\phi(x)=\phi_{d}(x)$. 

We define the residual at the point $(t_{i},x_{j})$ as
\begin{align*}
r_{i,j}=u(t_{i},x_{j};h,\hat{\theta}_{OLS}^{M,N}(h))-y_{ij}.
\end{align*}
By minimizing the numerical OLS cost function from Equation (\ref{eq:OLS_cost_num})
in our inverse problem methodology, we are implicitly assuming that
each residual value is independent and identically distributed (\emph{i.i.d.}),
which we expect to be true based on our statistical model in Equation
(\ref{eq:data_form}). We observe from Figure ~\ref{fig:ols_resid_2}
that the residuals are neither independent nor identically distributed:
they are largest near the front location and are correlated with their
neighboring residual values. Numerical diffusion from the upwind method
is the likely explanation for these residual patterns. It smoothens
the numerical solution near the point of discontinuity, which causes
the computation to fall below the analytical solution at values just
left to the point of discontinuity and to rise above the analytical
solution at values to the right of the point of discontinuity. The
correlation between neighboring data points indicates that an autocorrelative
statistical model may be suitable to describe this behavior.

\begin{figure}
\centering{}\includegraphics[width=0.79\textwidth]{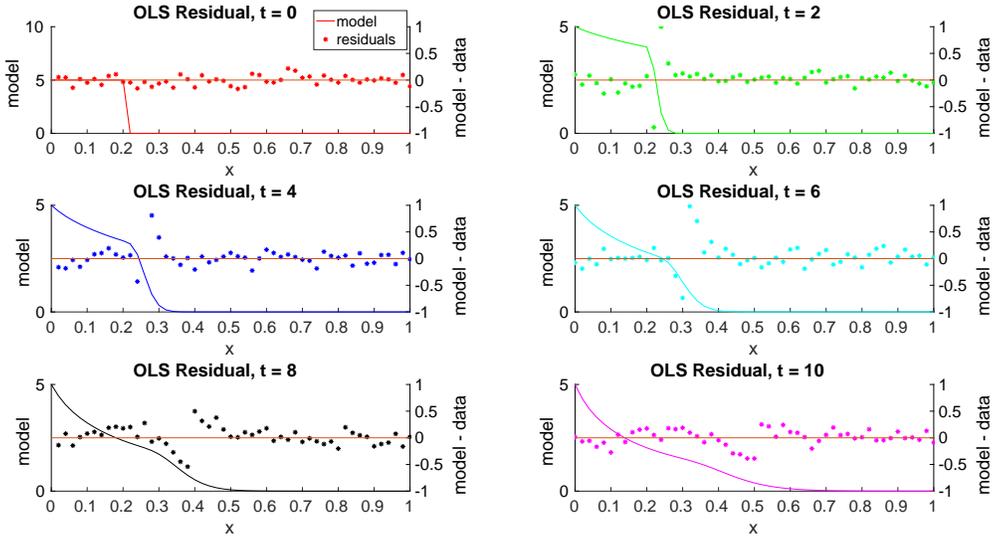}\caption{Plots of $r_{i,j}$ (dots) against simulations of $u(x,t;h,\hat{\theta})$
for the upwind method with $h=1/(10\times2^{4})$ for $\eta=0.1$
and $\phi(x)=\phi_{d}(x)$. \label{fig:ols_resid_2}}
\end{figure}

To quantify the autocorrelated error that arises from numerical diffusion
in this method, we assume the first order autocorrelation structure
from ~\citep[$\S$ 6.2.3]{seber_nonlinear_1988} arises. To illustrate
this structure, assume that the point of discontinuity occurs at the
location $x=x_{d_{i}}$ at $t=t_{i}$. This method assumes that the
residual values to the right of $x_{d_{i}}$ at the fixed time $t_{i}$
will satisfy
\begin{align}
r_{i,d_{i}} & =\dfrac{1}{\sqrt{1-\gamma_{i,+}^{2}}}\epsilon_{i,d_{i}}\nonumber \\
r_{i,d_{i+1}} & =\gamma_{i}^{+}r_{i,d_{i}}+\epsilon_{i,d_{i+1}}\nonumber \\
r_{i,d_{i+2}} & =\gamma_{i}^{+}r_{i,d_{i+1}}+\epsilon_{i,d_{i+2}}\nonumber \\
\vdots\nonumber \\
r_{iN} & =\gamma_{i}^{+}r_{iN-1}+\epsilon_{iN}\label{eq:auto_model}
\end{align}
for $\epsilon_{i,j}\stackrel{i.i.d.}{\sim}\mathcal{N}(0,\eta^{2})$
and $\gamma_{i}^{+}$ is the autocorrelation constant at time $t_{i}$
for points to the right of $x_{d}$. If we let $\vec{r}_{i,+}$ and
$\vec{\epsilon}_{i,+}$ denote the $(N-d_{i}+1)\times1$ vector of
spatial residual values and Gaussian noise terms to the right of (and
including) $x_{d}$ at time $t=t_{i}$, then 
\begin{equation}
R_{i}^{+}\vec{r}_{i,+}=\vec{\epsilon}_{i,+}\label{eq:auto_model_matrix_form}
\end{equation}
 for 
\begin{align*}
R_{i}^{+}=\left[\begin{array}{ccccc}
\sqrt{1-(\gamma_{i}^{+})^{2}} & 0 & 0 & \dots & 0\\
-\gamma_{i}^{+} & 1 & 0 &  & 0\\
0 & -\gamma_{i}^{+} & \ddots & \ddots & \vdots\\
\vdots &  & \ddots & 1 & 0\\
0 & 0 & \dots & -\gamma_{i}^{+} & 1
\end{array}\right].
\end{align*}
By combining Equations (\ref{eq:data_form}) and (\ref{eq:auto_model_matrix_form}),
we see that
\begin{equation}
R_{i}^{+}\vec{r}_{i,+}\stackrel{i.i.d.}{\sim}\mathcal{N}(0,\eta^{2}I).\label{eq:num_stat_model2}
\end{equation}
We will define an analogous statistical model at time $t_{i}$ for
the points to the left of $x=x_{d}$ with rate of autocorrelation
$\gamma_{i}^{-}$ and matrix $R_{i}^{-}$ so that $R_{i}=\text{diag}(\{R_{i}^{-},R_{i}^{+}\})$
and $R_{i}\vec{r}_{i}\stackrel{i.i.d.}{\sim}\mathcal{N}(0,\eta^{2})$
for $\vec{r}_{i}$ denoting the $N\times1$ vector of residuals at
time $t_{i}$. Ultimately, we have 
\begin{equation}
R\vec{r}\ \stackrel{i.i.d.}{\sim}\mathcal{N}(0,\eta^{2}I)\label{eq:num_stat_model3}
\end{equation}
for $R=\text{diag}\left(\{R_{i}\}_{i=1}^{M}\right)$ when $U(h,\theta)$
is used to approximate $U_{0}(\theta)$.

To estimate $\theta_{0}$ and quantify numerical error with an autocorrelation
model, we perform the following two-stage estimation routine for a
data set with a given step size, $h$ (taken from ~\citep[$\S$ 6.2.3]{seber_nonlinear_1988},
with modification):

\noindent\fbox{\begin{minipage}[t]{1\columnwidth - 2\fboxsep - 2\fboxrule}%
1. Fit the model by finding the estimator, $\hat{\theta}_{OLS}^{M,N}(h)$,
that minimizes Equation (\ref{eq:OLS_cost_num}).

2. Compute the corresponding OLS residuals, $\vec{r},$ and estimate
$\gamma_{i}$ and $\gamma_{i,b}$ using the formulas
\begin{align*}
\gamma_{i}^{+}=\dfrac{\sum_{j=d_{i}}^{N-1}r_{i,j}r_{i,j+1}}{\sum_{j=d_{i}}^{N-1}r_{i,j}^{2}},\ \gamma_{i}^{-}=\dfrac{\sum_{j=1}^{d_{i}-1}r_{i,j}r_{i,j+1}}{\sum_{j=1}^{d_{i}-1}r_{i,j}^{2}},\ i=1,...,M
\end{align*}

3. Fit the model by find the estimator, $\hat{\theta}_{auto}^{M,N}(h)$,
that minimizes 
\begin{align*}
J_{auto}^{M,N}(h,\theta)=\dfrac{1}{MN}\vec{r}^{T}V^{-1}\vec{r},\ V^{-1}=R^{T}R.
\end{align*}
\end{minipage}}

We performed this autocorrelation optimization method for the upwind
method and depict the resulting modified residuals, $R\vec{r}$, in
Figure ~\ref{fig:auto_resid_2}. Here we see that the modified residuals
do appear \emph{i.i.d.}, suggesting that the autocorrelation method
is capable of accurately correcting residual computations when error
from numerical diffusion arises. We only show the results for one
data set here, but others exhibit similar results.

\begin{figure}
\centering{}\includegraphics[width=0.79\textwidth]{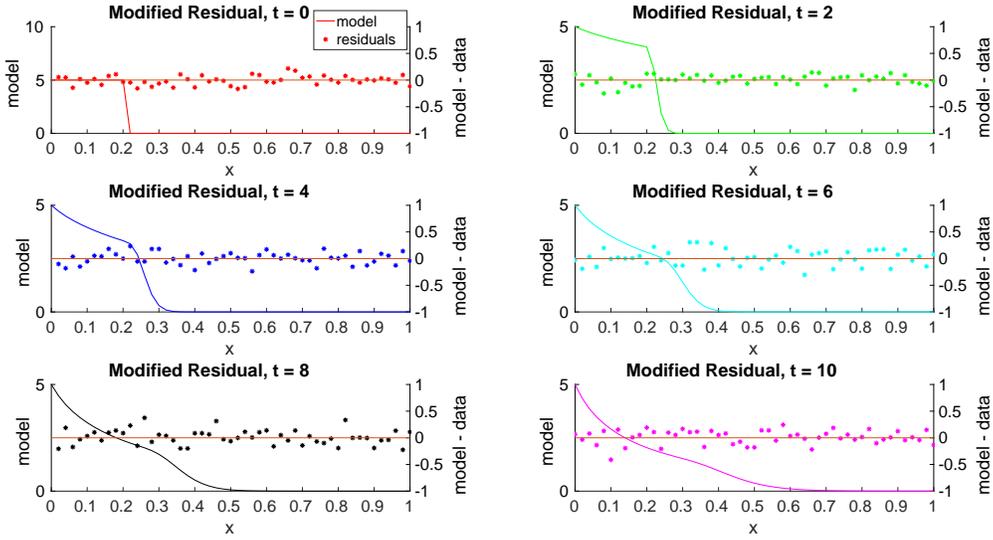}\caption{Plots of modified residuals (dots) against simulations of $u(x,t;h,\hat{\theta})$
for the upwind method with $h=1/(10\times2^{4})$ with $\eta=0.1$
and $\phi(x)=\phi_{d}(x)$.\label{fig:auto_resid_2}}
\end{figure}

The goal of the autocorrelative statistical model is not only to determine
the underlying statistical model, but also to improve estimation of
$\theta_{0}$ by doing so. In Figure ~\ref{fig:Plots-of-q_main_text},
we depict some plots of $\|\hat{\theta}_{auto}^{M,N}(\boldsymbol{h})-\theta_{0}\|_{2}$.
In Figure S13 in the supporting material, we show this for all data
sets considered. Here we see that $\hat{\theta}_{auto}^{M,N}(h)$
is improved over $\hat{\theta}_{OLS}^{M,N}(h)$ for the upwind method
for many data sets and step size values. This method even outperforms
the Beam-Warming, Lax-Wendroff, and upwind scheme with flux limiters
in several cases. For larger values of $\eta^{2},$ estimation of
$\theta_{0}$ was not significantly improved with the autocorrelation
estimation routine, suggesting that the autocorrelation scheme cannot
improve estimation when there is significantly more experimental error
present than numerical error.

\subsection{Confidence Interval Computation\label{subsec:Confidence-Interval-Computation}}

If for some matrix, $Q$, we let the estimator, $\hat{\theta}^{M,N}$,
satisfy 
\begin{align*}
\hat{\theta}^{M,N}(h)=\arg\min_{\theta\in Q_{ad}}\dfrac{1}{MN}\vec{r}^{T}Q^{T}Q\vec{r},
\end{align*}
and assume that the residuals satisfy $(Q\vec{r})_{i}\stackrel{i.i.d.}{\sim}\mathcal{N}(0,\eta^{2})$,
then asymptotically as $M,N\rightarrow\infty,$
\begin{align*}
\hat{\theta}^{M,N}(h)\sim\mathcal{N}(\theta_{0},H_{0}^{M,N})\approx\mathcal{N}\left(\theta_{0},\eta^{2}\left[\left[Q\nabla U_{0}(\theta_{0})\right]^{T}\left[Q\nabla U_{0}(\theta_{0})\right]\right]^{-1}\right).
\end{align*}
See ~\citep[Theorem 2.1]{seber_nonlinear_1988} for more details. Observe
that $P=I$ when minimizing the OLS cost function and $P=R$ when
minimizing the autocorrelation cost function described in Section
~\ref{subsec:Residual-Analysis}. From this, we can show that the $(1-a)100\%$
confidence interval for the $k^{\text{th}}$ component of $\theta_{0}$
is given by the interval
\begin{align}
\hat{\theta}_{k}^{M,N} & \pm SE_{k}(\hat{\theta}^{M,N})t_{1-\nicefrac{a}{2}}^{MN-k_{\theta}},\nonumber \\
\text{for }SE_{k}(\hat{\theta}) & =\sqrt{\hat{\eta}^{2}\hat{H}_{kk}(\hat{\theta})},\nonumber \\
\text{where }\hat{H}(\hat{\theta}) & =\left[\left(Q\nabla U_{0}(\theta_{0})\right)^{T}\left(Q\nabla U_{0}(\theta_{0})\right)\right]^{-1}\nonumber \\
\text{and }\hat{\eta}^{2} & =\dfrac{1}{MN-k_{\theta}}\vec{r}^{T}Q^{T}Q\vec{r}\label{eq:OLS_CI}
\end{align}
where $t_{1-\nicefrac{a}{2}}^{n}$ is the value such that $P(T\ge t_{1-\nicefrac{a}{2}}^{n})=\nicefrac{a}{2}$
if $T$ is a sample from the student's t-distribution with $n$ degrees
of freedom.

In Figure ~\ref{fig:95=000025confidence_UW_OLS}, we depict several
95\% OLS confidence intervals that have been computed with the upwind
method for $\phi(x)=\phi_{d}(x)$. The blue (red) confidence regions
have been computed with large (small) values of $h$. We observe that
the confidence intervals can enclose $\theta_{0}$ well for $N=11$
with the finest grid computations, but often miss $\theta_{0}$ for
$N=30$. Note that these confidence intervals for $N=30$ are close
to $\theta_{0},$ yet their small areas prevent them from actually
enclosing $\theta_{0}$. In the supporting material, we depict the
confidence intervals for all data sets considered for $\phi(x)=\phi_{c}(x)$
in Figures S8-S11. The upwind scheme struggles in these confidence
intervals, but the Beam-Warming and Lax-Wendroff schemes can enclose
$\theta_{0}$ reliably. In the supporting material, we depict the
confidence intervals for all data sets considered for $\phi(x)=\phi_{d}(x)$
in Figures S18-S21. The confidence regions for the Lax-Wendroff, Beam-Warming,
and upwind with flux limiters methods can all capture $\theta_{0}$
for $N=11,$ but struggle for $N=30$ and 51. These confidence intervals
approach $\theta_{0}$ as $h\rightarrow0$, but their areas are too
small to capture $\theta_{0}$.

\begin{figure}
\includegraphics[width=0.99\textwidth]{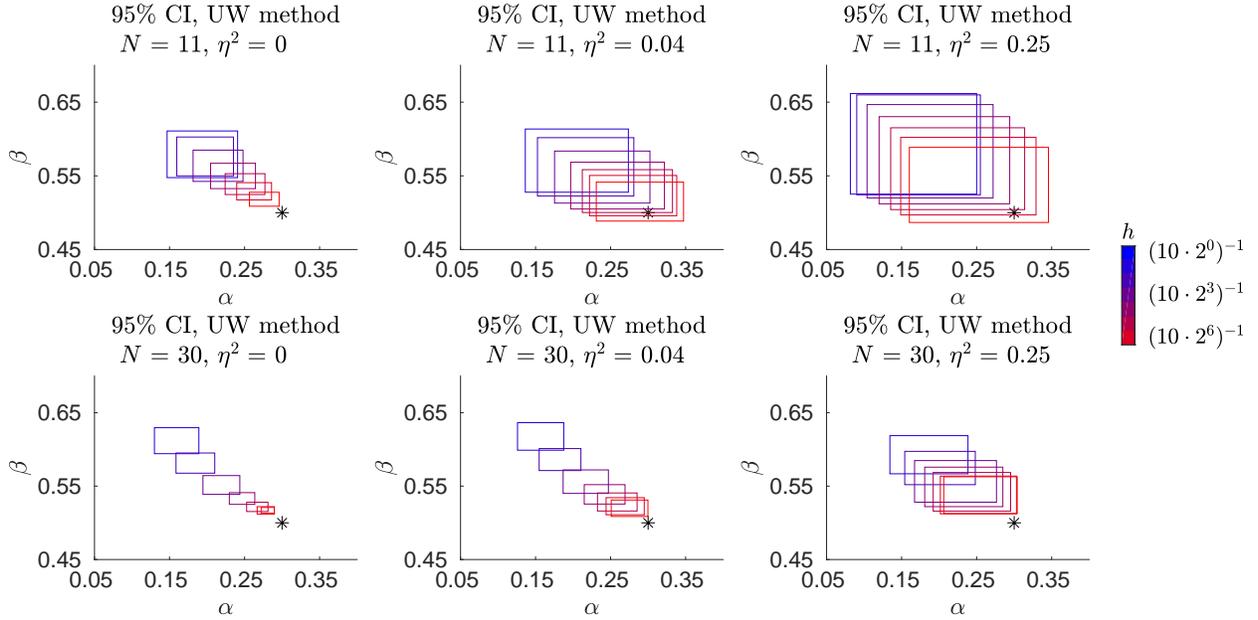}\caption{95\% OLS confidence intervals for $\theta=(\alpha,\beta)^{T}$ using
Equation (\ref{eq:OLS_CI}) with an upwind scheme and $\phi(x)=\phi_{d}(x)$
. The asterisk denotes $\theta_{0},$ and computations were done with
a smaller value of $h$ as the confidence region color changes from
blue to red.\label{fig:95=000025confidence_UW_OLS}}
\end{figure}

\begin{figure}
\includegraphics[width=0.99\textwidth]{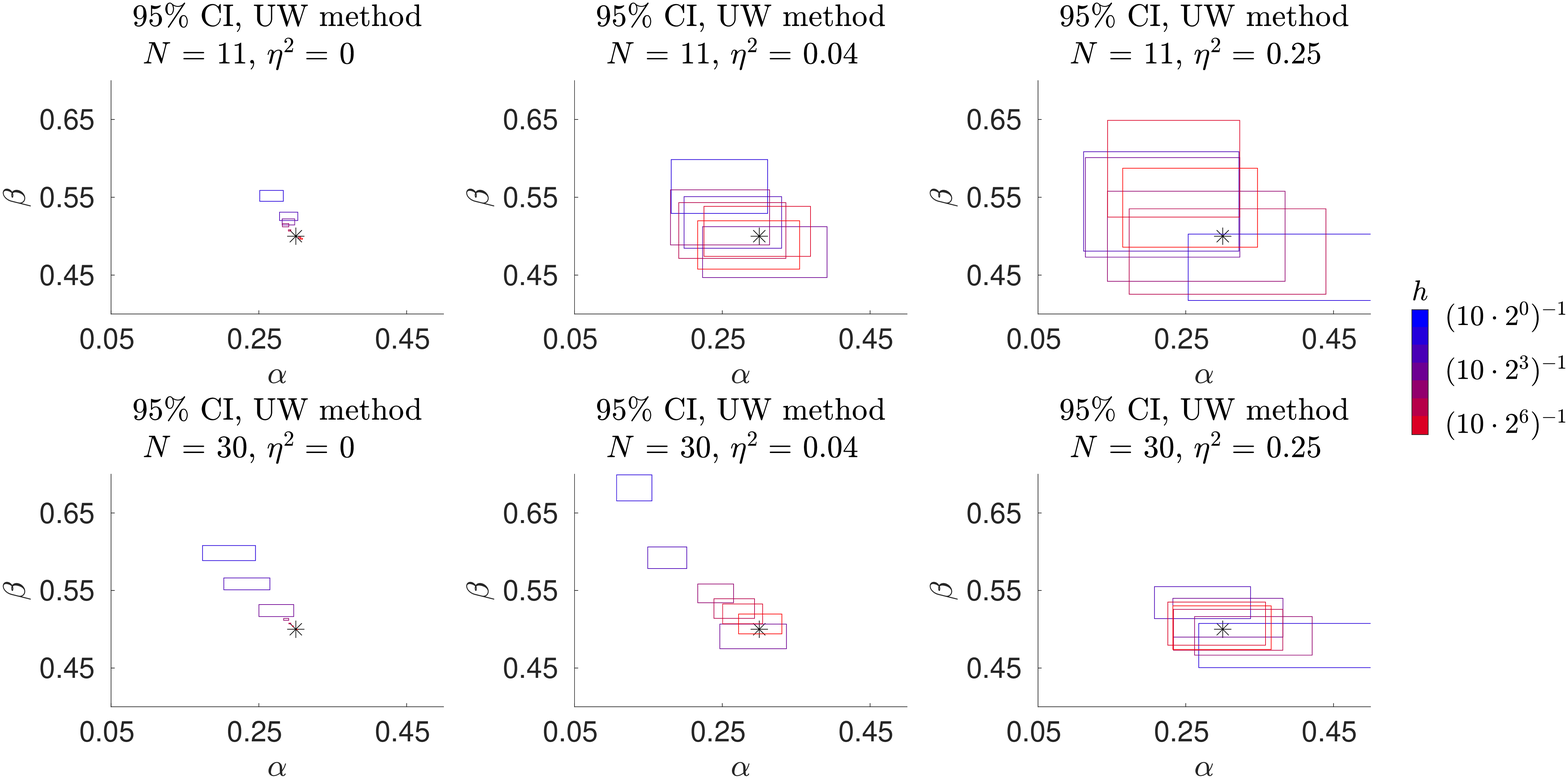}\caption{95\% autocorrelative confidence intervals for $\theta=(\alpha,\beta)^{T}$
using Equation (\ref{eq:OLS_CI}) with an upwind scheme and $\phi(x)=\phi_{d}(x)$.
The asterisk denotes $\theta_{0},$ and computations were done with
a smaller value of $h$ as the confidence region color changes from
blue to red. \label{fig:95=000025-autocorrelative-confidence_upwind}}
\end{figure}

Figure ~\ref{fig:95=000025-autocorrelative-confidence_upwind} depicts
the calculated 95\% confidence intervals for $\theta=(\alpha,\beta)^{T}$
when using the autocorrelative statistical model from Section ~\ref{subsec:Residual-Analysis}
with the upwind method and $\phi(x)=\phi_{d}(x)$. We see that these
confidence regions are an improvement over the OLS confidence intervals,
as the confidence intervals enclose $\theta_{0}$ for most values
of $h$ when $N=11$ and $\eta^{2}\ne0,$ and the confidence intervals
do enclose $\theta_{0}$ for smaller values of $h$ when $N=30$ and
$\eta^{2}\ne0.$ The autocorrelative confidence intervals are depicted
for all data sets in the supporting material in Figure S22. In general,
the method can significantly improve confidence interval computation
for the upwind scheme, but still struggles when $N=51$.

\section{Suggestions for practitioners\label{sec:Suggestions-for-practitioners}}

Based on our results, we suggest some strategies for practitioners
in this section to improve their inverse problem methodologies. The
conclusions from this section are summarized in Table ~\ref{tab:Summary-of-strategies}.

If one is concerned with minimizing $J_{OLS}^{M,N}(h,\hat{\theta}_{OLS}^{M,N}(h))$
(and in turn inferring\emph{ }$\eta^{2}$, the variance of the experimental
error in their data\footnote{assuming the statistical model provided in Equation (\ref{eq:data_form})
is accurate. If not, a slightly modified cost function can be used
to infer $\eta^{2}$ ~\citep[$\S$ 3.2]{banks_mathematical_2009}. Residuals
are a useful tool for determining the underlying statistical model
~\citep{banks_extension_2012}. Different types of statistical models
are discussed in length in ~\citep{seber_nonlinear_1988}. }), one can determine if they have reached the true minimum value by
performing the inverse problem discussed here for multiple values
of the grid size, $h$. If the computed cost function decreases as
$h$ decreases, then $J_{OLS}^{M,N}(h,\hat{\theta}_{OLS}^{M,N}(h))$
is likely larger than $\eta^{2}$ and not a reliable estimate. In
this case, computation of $J_{OLS}^{M,N}(h,\hat{\theta}_{OLS}^{M,N}(h))$
can be improved with further grid refinement or by quantifying the
effects of numerical error through a statistical model, similar to
our analysis in Section ~\ref{sec:Residual-Analysis}. If the order
of the numerical cost function appears to be zero, which can be confirmed
by finding the best-fit line to $\ln\left(J_{OLS}^{M,N}(\boldsymbol{h},\hat{\theta}_{OLS}^{M,N}(\boldsymbol{h}))\right)$
against $\ln(\boldsymbol{h})$, then the practitioner can be confident
that $J_{OLS}^{M,N}(h,\hat{\theta}_{OLS}^{M,N}(h))\approx\eta^{2},$
especially if $M,N$ are large. More data points can also make the
computation of $J_{OLS}^{M,N}(h,\hat{\theta}_{OLS}^{M,N}(h))$ as an 
estimate of $\eta^2$ more reliable.

We observe in Figure ~\ref{fig:Plots-of-q_main_text} that the choice
of step size, $h$, and numerical method can lead to different parameter
estimate values. Accurate parameter estimation is a crucial element
in understanding the scientific system under consideration. To determine 
which numerical method can most accurately estimate $\theta_{0}$ for a 
given $h$, a practitioner may use an artificially-generated data set,
 similar to what we've done in this study. This data set should
resemble the true data as much as possible: it should have the same
number of data points, be parameterized by some rough estimate of
$\theta_{0}$ (such as $\hat{\theta}_{OLS}^{M,N}(h)$ for some value
of $h$), and have the variance of the data points be an estimate
of $\eta^{2}$ (such as $J_{OLS}^{M,N}(h,\hat{\theta}_{OLS}^{M,N}(h))$
for some value of $h$). If an analytical solution is not available
for this data generation, then a very small value of $h$ could be
used to generate the data, and a very accurate numerical method (such
as the upwind scheme with flux limiters) should be used. One should
be mindful that the choice of numerical method may skew their results.
With this data set, determine which numerical method can most accurately
estimate the parameter value used to parameterize the artificial data.
This method should be used to fit the experimental data and estimate
$\theta_{0}$. 

Using a smaller value of $h$ will often not be a practical solution
as a means to improve inverse problem results. We saw in this work
that two alterations can be incorporated with the upwind method to
improve its results: the use of flux limiters in computation or an
autocorrelative statistical model. Both of these strategies have their
advantages and disadvantages. The upwind scheme with flux limiters
yields a very accurate simulation profile (as seen in Figure ~\ref{fig:Numerical-solution-profiles}),
but does increase the computation time because the spatial gradient
has to be estimated at each time iteration. The autocorrelative statistical
model is not computationally expensive; it should only double the computation time of the inverse problem (one round of OLS optimization followed by another
round of autocorrelative optimization). We see in Figure S22 in the
supporting material that this autocorrelative statistical model can 
successfully improve estimate values of $\theta_{0}$ 
when $\eta^{2}$ is small. Both of these alterations
were also only effective for $\phi(x)=\phi_{d}(x);$ we can not recommend
one use these methods when the model solution is continuous. 

Lastly, we saw that both the incorporation of flux limiters and the
autocorrelative statistical method could enhance confidence interval
computation in Section ~\ref{subsec:Confidence-Interval-Computation}.
The biggest factor in preventing accurate confidence interval computation
in this study is large numbers of data points. All methods struggled
to enclose $\theta_{0}$ for $N=51$, which is likely because these
confidence intervals have very small areas. If a practitioner is concerned
with accurate confidence interval computation, then we may suggest
checking that they can accurately enclose $\theta_{0}$ for artificial
data sets with the same number of points as their data sets. If not,
they should consider subsets of their data that will compute wider 
confidence regions that can capture $\theta_{0}$ more reliably.
The autocorrelative statistical model and the incorporation of 
flux limiters are also excellent methods to improve confidence 
interval computation, as seen in Figures S21 and S22 in the supporting
 material.

\begin{sidewaystable}
\vspace{15cm}\hspace{-3.25cm}\begin{tabular}{|c|c|c|}
\hline 
\textbf{Task} & \textbf{To do} & \textbf{Conclusions/Notes}\tabularnewline
\hline 
Improve &  & If $J_{OLS}(\hat{\theta},\boldsymbol{h})$ does not change, the minimum
has likely been reached.\tabularnewline
minimization  & compute $J_{OLS}(\hat{\theta},h)$ & If $J_{OLS}(\hat{\theta},\boldsymbol{h})$ is decreasing with $h$,
computation can be improved with \tabularnewline
of $J_{OLS}(\hat{\theta},h)$ & for several values of $h$ & smaller values of $h$ or by using a statistical model to account
for \tabularnewline
 &  & numerical error.\tabularnewline
\hline 
Determine best & Perform IP on & Choose the method that can best predict the value of $\theta$ that \tabularnewline
numerical method & artificially-generated data & generated the data. If no analytical solution exists, keep in mind\tabularnewline
 & with multiple methods & that the method used in generating data will bias results.\tabularnewline
\hline 
Improve results & 1. Perform more accurate & 1. Note that flux limiters will also increase the computation time.\tabularnewline
without  & computation (\emph{e.g.}, flux limiters) & 2. Note that the autocorrelative statistical model worked best\tabularnewline
decreasing $h$ & 2. Use statistical model & when $\eta^{2}$ was small.\tabularnewline
 & to incorporate numerical error & Both of these strategies work well for discontinuous solutions.\tabularnewline
\hline 
 & 1. Perform IP on & 1. Reduce number of data points\lyxdeleted{John Nardini,,,}{Sun Jul  8 21:45:55 2018}{
} to a number\tabularnewline
 & artificially-generated data & where CI computations reliably enclosed $\theta_{0}$ \tabularnewline
Improve CI & with fewer data points. & for artificially generated data sets. \tabularnewline
computation & 2. Use statistical model & 2. Note that this only works for methods that\tabularnewline
 & to incorporate numerical error & admit numerical diffusion.\tabularnewline
 & 3. Use flux limiters & 3. Will increase computation time.\tabularnewline
\hline 
\end{tabular}\caption{Summary of strategies to help practitioners improve the results of
their inverse problem methodologies. The ``Task'' column denotes
a task that one may wish to carry out. The ``To do'' column suggests
some strategies to perform the desired
task. The ``Conclusions/Notes'' column provides guidelines on how
to interpret the different results one may find, as well as notes
to keep in mind when making final conclusions. The abbreviations in
the table include: $\hat{\theta}=\hat{\theta}_{OLS}^{M,N}(h),$ IP
= inverse problem, and CI = confidence interval. \label{tab:Summary-of-strategies}}
\end{sidewaystable}

\section{Discussion and Future Work \label{sec:Discussion-and-Future}}

Numerical approximations for advection-dominated processes are a known
challenge in the sciences ~\citep{leonard_ultimate_1991,thackham_computational_2008},
and the precise effects of numerical error on an inverse problem have
not been investigated thoroughly. In this document, we fit various
numerical schemes with varying orders of convergence to artificial
data with different numbers of data points and error levels. We use
a numerical cost function in a similar vein to that in ~\citep{banks_statistical_1990}
to show how the convergence of the cost function depends on the orders
of convergence of the numerical scheme used. We also determined the asymptotic
distribution of the parameter estimator in the presence of approximation error. 
In general, the second
order methods outperform the first order upwind methods in computing
the cost function and in parameter estimation, as one would would
expect. There are ways to improve results with this first order method,
however, including the use of flux limiters or an autocorrelative
statistical model. This autocorrelative statistical model describes
how numerical error propagates in the inverse problem and in turn
improve parameter estimation and confidence interval computation.
The incorporation of flux limiters into computation with the upwind
method improves computation accuracy as well as parameter estimation.

There are some aspects of this study that we have left for future
work. In Figure ~\ref{fig:LW_front_res_OLS}, we depict the OLS residuals
when fitting the Lax-Wendroff method to the artificial data when $\phi(x)=\phi_{d}(x)$.
Recall that the modified equation for this second order method is
dispersive, so the leading error terms are composed of high-frequency
modes from the initial condition propagating at different speeds.
This set of residuals shows patterns that would be much more difficult
to quantify than those presented in Section ~\ref{subsec:Residual-Analysis}.
Future work should include a careful analysis into how numerical error
from this and other higher-order numerical methods influences the
statistical model of the data. As we saw in this work, determining
this influence would lead to improvements in parameter estimation
and uncertainty quantification for these methods.

\begin{figure}
\centering{}\includegraphics[width=0.79\textwidth]{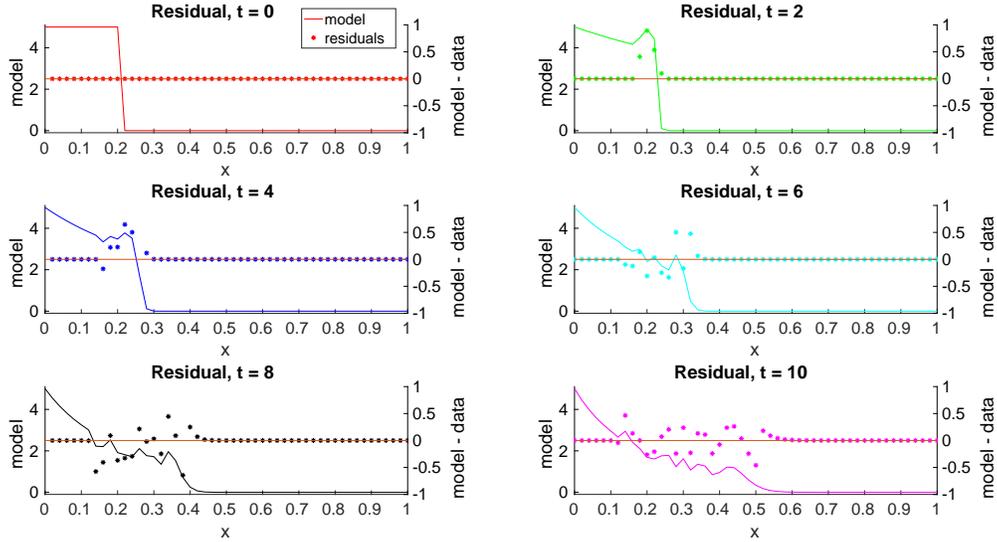}\caption{Plots of $r_{i,j}$ (dots) against simulations of $\vec{u}(x,t;h,\hat{\theta}_{OLS}^{M,N}(h))$
for the Lax-Wendroff method with $h=1/(10\times2^{6})$ with $\eta=0$
for $\phi(x)=\phi_{d}(x).$ \label{fig:LW_front_res_OLS}}
\end{figure}

\ack

This work is supported in part by the Joint NSF/NIH Mathematical Biology Initiative Program via grant NIGMS-R01GM126559.

\appendix

\section{Previous Theory of $\hat{\theta}_{OLS}^{M,N}$\label{sec:Theory-on-theta_ols_h}}

In this Section, we state part of Theorem 2.1 from ~\citep{seber_nonlinear_1988}.
\begin{thm*}
Given that $\mathcal{\vec{\epsilon}\sim\mathcal{N}}(0,\eta^{2}I)$
and $u_{0}(t,x;\theta)$ is sufficiently smooth with respect to $\theta,$
then we have the asymptotic distribution for $\theta_{OLS}^{M,N}$
as $M,N\rightarrow\infty$ given by 
\begin{align*}
\theta_{OLS}^{M,N}\sim\mathcal{N}(\theta_{0},\eta^{2}V),\ V=(\nabla_{\theta}U_{0}(\theta_{0})^{T}\nabla_{\theta}U_{0}(\theta_{0}))^{-1}.
\end{align*}
 
\end{thm*}

\section{Convergence of the terms of $J_{OLS}^{M,N}(h,\theta)$\label{sec:Convergence-of-the}}

Here we discuss the asymptotic properties of $J_{OLS}^{M,N}(h,\theta)$.
We begin with the limits as $h\rightarrow0$ in Section ~\ref{subsec:Limits-ash}
and as $M,N\rightarrow\infty$ in ~\ref{subsec:Limits-asMN}.

For brevity, we denote $\hat{\theta}_{OLS}^{M,N}(h)$ as $\hat{\theta}$
for the rest of this section. 

Note that ~\citep{banks_statistical_1990} includes more assumptions
than those provided in this study, but they are already satisfied
by Equations (\ref{eq:adv_eqn}) or (\ref{eq:data_form}). Assumption
(A1) here is a modification of (A3) in ~\citep{banks_statistical_1990}
to include convergence of $L^{1}$ functions. 

\subsection{Limits as $h\rightarrow0$\label{subsec:Limits-ash}}

Here we provide the proof to Lemma ~\ref{lem:num_conv}.
\begin{proof}
Note that our definition for the numerical order of convergence gives
that 
\begin{align*}
\|\vec{U}_{0}(\hat{\theta})-\vec{U}(h,\hat{\theta})\|_{1}=\mathcal{O}(h^{p}).
\end{align*}

Term A is independent of $h$, so it acts as $\mathcal{O}(1)$.

Term $B$ is given as
\begin{align*}
B=\dfrac{1}{MN}\sum_{i,j=1}^{M,N}[u_{0}(t_{i},x_{j};\theta_{0})-u_{0}(t_{i},x_{j};\hat{\theta})]^{2}.
\end{align*}
As $h\rightarrow0$ and $\hat{\theta}$ approaches $\theta_{0}$ (assuming
that there are enough data points used for this to occur). We can
Taylor expand about $\theta_{0}$ and find
\begin{align*}
B\approx\dfrac{1}{MN}\sum_{i,j=1}^{M,N}[\nabla_{\theta}U_{0}(t_{i},x_{j};\theta_{0})[\hat{\theta}-\theta_{0}]]^{2}.
\end{align*}
Note that $\nabla_{\theta}U_{0}(t_{i},x_{j};\theta_{0})$ is independent
of $h$, but from Corollary ~\ref{cor:2}, we have that $\hat{\theta}\sim\mathcal{N}(\theta_{0},V_{h})$
where each entry of $V_{h}$ converges to its corresponding entry
of $V$ as $\mathcal{O}(h^{p}).$ Each term being summed is thus a
random variable with mean independent of $h$ and variance acting
as $\mathcal{O}(h^{2p}).$ We thus conclude that this random variable
has standard deviation $\mathcal{O}(h^{p})$. Thus $B$ converges as $\mathcal{O}(h^{p})$.

Term $C$ is given by 
\begin{align*}
C=\dfrac{1}{MN}\sum_{i,j=1}^{M,N}[u_{0}(t_{i},x_{j};\hat{\theta})-u(t_{i},x_{j};h,\hat{\theta})]^{2},
\end{align*}
which may also be written in terms of the Euclidean vector norm, from
where we can then use equivalence of finite-dimensional norms to show that it will
converge as $\mathcal{O}(h^{2p})$ by assuming that $\hat{\theta}$
is in the compact space, $\Theta_{ad}$:

\begin{align*}
C & =\dfrac{1}{MN}\|\vec{U}_{0}(\hat{\theta})-\vec{U}(h,\hat{\theta})\|_{2}^{2}\\
 & \le\dfrac{K^{2}}{MN}\|\vec{U}_{0}(\hat{\theta})-\vec{U}(h,\hat{\theta})\|_{1}^{2}=K_Ch^{2p}\text{.}
\end{align*}
Thus $C$ converges as $\mathcal{O}(h^{2p})$.

Term $D$ is given by 
\begin{align*}
D=\dfrac{2}{MN}\sum_{i,j=1}^{M,N}\mathcal{\epsilon}_{i,j}(u_{0}(t_{i},x_{j};\theta_{0})-u_{0}(t_{i},x_{j};\hat{\theta}))
\end{align*}
 We begin with a Taylor expansion about $\theta_{0}$ to find
\begin{align*}
D\approx\dfrac{-2}{MN}\sum_{i,j=1}^{M,N}\epsilon_{i,j}(\nabla_{\theta}u_{0}(t_{i},x_{j};\theta_{0})(\hat{\theta}-\theta_{0})).
\end{align*}
We then use the Cauchy-Schwarz Inequality to show
\begin{align*}
|D|\lesssim\sqrt{\dfrac{2}{MN}\left(\sum_{i,j=1}^{M,N}\epsilon_{i,j}^{2}\right)\dfrac{2}{MN}\left(\sum_{i,j=1}^{M,N}(\nabla_{\theta}U_{0}(\theta_{0})(\hat{\theta}-\theta_{0}))^{2}\right)}.
\end{align*}
The first term on the right will be close to its finite mean of $\sqrt{2\eta^{2}}$
if $M,N$ are large by the law of large numbers (LLN). By Corollary
~\ref{cor:2}, the second term on the right is equivalent to
\begin{align*}
\sqrt{\dfrac{2}{MN}\left(\sum_{i,j=1}^{M,N}(\nabla_{\theta}U_{0}(\theta_{0})\theta_{D})^{2}\right)},
\end{align*}
where $\theta_{D}$ has a standard deviation that converges to $V$
as $\mathcal{O}(h^{p/2}).$ Everything else in this term is independent
of $h$, so $D$ is a random variable with standard
deviation converging as $\mathcal{O}(h^{p/2})$.
Thus $D$ converges as $\mathcal{O}(h^{p/2})$.

The term $E$ is written as
\begin{align*}
E=\dfrac{2}{MN}\sum_{i,j=1}^{M,N}\mathcal{\epsilon}_{i,j}(u_{0}(t_{i},x_{j};\hat{\theta})-u(t_{i},x_{j};h,\hat{\theta})).
\end{align*}
We can bound this term from above as $h\rightarrow0$ as 
\begin{align*}
|E| & \le\dfrac{2}{MN}\sqrt{\left(\sum_{i,j=1}^{M,N}\epsilon_{i,j}^{2}\right)\left(\sum_{i,j=1}^{M,N}\left|(u_{0}(x_{i},t_{j};\theta)-u(x_{i},t_{j};h,\theta))\right|^{2}\right)}\\
 & \le\dfrac{2K}{MN}\sqrt{\left(\sum_{i,j=1}^{M,N}\epsilon_{i,j}^{2}\right)\left(\sum_{i,j=1}^{M,N}\left|(u_{0}(x_{i},t_{j};\theta)-u(x_{i},t_{j};h,\theta))\right|\right)^{2}}\\
 & =\mathcal{O}(h^{p})
\end{align*}
where the first inequality is by the Cauchy-Schwarz inequality, the
second is by the equivalence of finite-dimensional norms. The final
approximation is from the LLN giving that the first term will converge
to its finite mean for $M,N$ large and then our definition for the
numerical order of convergence. Thus $E$ converges as $\mathcal{O}(h^{p})$.

Term F is written as
\begin{align*}
F=\dfrac{2}{MN}\sum_{i,j=1}^{M,N}\left[\left(u_{0}(x_{i},t_{j};\hat{\theta})-u(x_{i},t_{j};h,\hat{\theta})\right)\left(u_{0}(x_{i},t_{j};\theta_{0})-u_{0}(x_{i},t_{j};\hat{\theta})\right)\right].
\end{align*}
Using the Cauchy-Schwarz Inequality, we find

\begin{align*}
|F| & \le\dfrac{2}{MN}\|\vec{U_{0}}(\hat{\theta})-\vec{U}(h,\hat{\theta})\|_{2}\|\vec{U}_{0}(\theta_{0})-\vec{U}_{0}(\hat{\theta})\|_{2}.
\end{align*}
We then use the equivalence of norms and Taylor expansion about $\theta_{0}$
to find
\begin{align*}
|F| & \lesssim\dfrac{2K}{MN}\|\vec{U_{0}}(\hat{\theta})-\vec{U}(h,\hat{\theta})\|_{1}\|\nabla U_{0}(\theta_{0})(\hat{\theta}-\theta_{0})\|_{2}.
\end{align*}
The first term converges as $\mathcal{O}(h^{p})$ from our definition
for the numerical order of convergence. The second term is a random
variable with standard deviation converging as $\mathcal{O}(h^{p/2})$
from Corollary ~\ref{cor:2}. Thus $F$ converges as $\mathcal{O}(h^{3p/2})$.

\end{proof}

\subsection{Limits as $M,N\rightarrow\infty$\label{subsec:Limits-asMN}}

Here we provide the proof of Lemma ~\ref{lem:data_conv}.
\begin{proof}
Note that $\epsilon_{i,j}^{2}$ is distributed as $\eta^{2}$ times
a degree-1 chi-squared random variable. We thus observe that $A$
is distributed as $\nicefrac{\eta^{2}}{MN}$ times a degree-$MN$
chi-squared random variable, which has mean $\eta^{2}$ and variance
$\nicefrac{2\eta^{4}}{MN}.$ By the classical Central Limit Theorem
(CLT), 
\begin{align*}
\sqrt{MN}(A-\eta^{2})\xrightarrow{D}\mathcal{N}(0,2\eta^{2})
\end{align*}
as $M,N\rightarrow\infty$, where $\xrightarrow{D}$ denotes convergence
in distribution. Thus $A$ converges as
$\mathcal{O}\left(1/\sqrt{MN}\right)$.

Term B is the sum of the difference of the true solution squared when
computed at $\theta_{0}$ and $\theta$. From assumption (A1),
\begin{align*}
B(\hat{\theta}) & =\dfrac{1}{MN}\sum_{i,j=1}^{M,N}[u_{0}(t_{i},x_{j};\theta_{0})-u_{0}(t_{i},x_{j};\hat{\theta})]^{2}\rightarrow J^{*}(\hat{\theta}).
\end{align*}
$\text{as }M,N\rightarrow\infty$ by (A1). This convergence is
identical to a first order Riemann sum. Thus $B$ converges
with order $\mathcal{O}\left(1/(MN)\right)$.

Term $C$ is given by 
\begin{align*}
C=\dfrac{1}{MN}\sum_{i,j=1}^{M,N}[u_{0}(t_{i},x_{j};\hat{\theta})-u(t_{i},x_{j};h,\hat{\theta})]^{2}.
\end{align*}
We assume $\hat{\theta}$ stays within $Q_{ad}$ and use our Definition 
for the order of convergence to find that
\begin{align*}
C\approx\dfrac{1}{MN}\sum_{i,j=1}^{M,N}[w(t_{i},x_{j};h)h^{p}]^{2}=\mathcal{O}(h^{2p}),
\end{align*}
Thus $C$ is independent of $M,N$.

Term D is written as
\begin{align*}
D=\dfrac{2}{MN}\sum_{i,j=1}^{M,N}\epsilon_{i,j}(u_{0}(t_{i},x_{j};\theta_{0})-u_{0}(t_{i},x_{j};\hat{\theta})).
\end{align*}
$u_{0}(t,x;\theta)$ is bounded below by 0 and above by 1, so we can
bound this term as
\begin{align*}
\dfrac{-2}{MN}\sum_{i,j=1}^{M,N}\epsilon_{i,j}\le D\le\dfrac{2}{MN}\sum_{i,j=1}^{M,N}\epsilon_{i,j}.
\end{align*}
By the CLT, both of these bounds will converge in distribution to
zero with order $\mathcal{O}(1/\sqrt{MN})$. Thus
$D$ converges as $\mathcal{O}(1/\sqrt{MN}).$

Term $E$ is written as
\begin{align*}
E=\dfrac{2}{MN}\sum_{i,j=1}^{M,N}\epsilon_{i,j}(u_{0}(t_{i},x_{j};\theta)-u(x_{i},t_{j};h,\theta)).
\end{align*}
We assume $\hat{\theta}$ stays within $Q_{ad}$ and use our definition
for the order of convergence to find that 
\begin{align*}
E\approx\dfrac{2}{MN}\sum_{i,j=1}^{M,N}\epsilon_{i,j}w(t_{i},x_{j};h)h^{p},
\end{align*}
so that
\begin{align*}
-\dfrac{2h^{p}}{MN}\sum_{i,j=1}^{M,N}\epsilon_{i,j}\lesssim E\lesssim\dfrac{2h^{p}}{MN}\sum_{i,j=1}^{M,N}\epsilon_{i,j},
\end{align*}
which shows that $E$ will converge in distribution to zero with
order $\mathcal{O}(1/\sqrt{MN})$ as $M,N\rightarrow\infty$ by the
CLT. Thus $E$ converges as $\mathcal{O}(1/\sqrt{MN})$.

Term F is written as
\begin{align*}
F=\dfrac{2}{MN}\sum_{i,j=1}^{M,N}\left[\left(u_{0}(x_{i},t_{j};\hat{\theta})-u(x_{i},t_{j};h,\hat{\theta})\right)\left(u_{0}(x_{i},t_{j};\theta_{0})-u_{0}(x_{i},t_{j};\hat{\theta})\right)\right].
\end{align*}
If we assume that $\hat{\theta}$ remains in $Q_{ad}$ then we can
use our definition for the order of convergence and the boundedness
of $w(t,x;h)$ to show that
\begin{align*}
\dfrac{-2h^{p}}{MN}\sum_{i,j=1}^{M,N}\left(u_{0}(x_{i},t_{j};\theta_{0})-u_{0}(x_{i},t_{j};\hat{\theta})\right)\le F\le\dfrac{2h^{p}}{MN}\sum_{i,j=1}^{M,N}\left(u_{0}(x_{i},t_{j};\theta_{0})-u_{0}(x_{i},t_{j};\hat{\theta})\right).
\end{align*}
If we define 
\begin{align*}
J^{1}(\theta)=\int_{X}\int_{T}\left(u_{0}(x,t;\theta_{0})-u_{0}(x,t;\theta)\right)d\nu(t)d\chi(x),
\end{align*}
then the sum in the above equation will converge to 
$\mathcal{O}(h^{p})J^{1}(\hat{\theta})$ as a first order Riemann sum.
Note that we can bound this integral between -10 and 10.
Thus $F$ converges as $\mathcal{O}(1/(MN))$ to a $\mathcal{O}(h^{p})$ term.
\end{proof}

\newcommand{\newblock}{}

\end{document}